\documentclass[11pt]{amsart}

\usepackage[a4paper, margin=2.5cm]{geometry}
\usepackage{multicol}			

\usepackage{hyperref}
\hypersetup{
    colorlinks,
    citecolor=green,
    filecolor=black,
    linkcolor=blue,
    urlcolor=black}

\usepackage{amssymb}				
\usepackage{mathrsfs}				
\usepackage{dsfont}	
\usepackage{xcolor}
\usepackage{bbm}
\usepackage{comment}

\newcommand{\ZZ}{\mathbb{Z}}
\definecolor{dgreen}{HTML}{026a10}
\definecolor{dviolet}{HTML}{9109E3}
\definecolor{dorange}{HTML}{e55700}

\usepackage{tikz}
\usetikzlibrary{matrix,arrows}
\usepackage{tikz-cd} 	
\usetikzlibrary{shapes.geometric}
\usepackage{graphicx}
\usepackage{subcaption}
\usepackage{adjustbox}
\usetikzlibrary{calc}
\usepackage{float} 

\usepackage{amsmath} 	
\usepackage{amsthm} 	
\usepackage{mathtools}	
\usepackage{array}
\def\beq{\begin{equation}}
\def\eeq{\end{equation}}
\newtheorem{thm}{Theorem}[section]
\newtheorem{prop}[thm]{Proposition}
\newtheorem{cor}[thm]{Corollary}
\newtheorem{lem}[thm]{Lemma}

\theoremstyle{definition}
\newtheorem{dfn}[thm]{Definition}
\newtheorem{ex}[thm]{Example}
\theoremstyle{remark}
\newtheorem*{rmk}{Remark}

\newcommand{\PP}{\mathbb{P}}
\newcommand{\RR}{\mathbb{R}}

\newcommand{\CC}{\mathbb{C}}

\newcommand{\A}{\mathcal{A}}
\newcommand{\B}{\mathcal{B}}
\newcommand{\I}{\mathcal{I}}
\newcommand{\D}{\mathcal{D}}
\newcommand{\Hom}{\ensuremath{\text{Hom}}}
\newcommand{\Newt}{\ensuremath{\text{Newt}}}
\newcommand{\trop}{\ensuremath{\text{trop}}}
\newcommand{\Cone}{\ensuremath{\text{Cone}}}

\allowdisplaybreaks

\graphicspath{ {tikz/} }

\title{Dyck Paths, Configuration Spaces and Polytopes For Linear Nakayama algebras}

\author{Veronica Calvo Cortes}
\address{Max Planck Institute for Mathematics in the Sciences\\ Leipzig, Germany}
\email{veronica.calvo@mis.mpg.de}

\author{Hadleigh Frost}
\address{Institute for Advanced Study,\\ Princeton, 08540 NJ, U.S.A.}
\email{frost@ias.edu}

\begin{document}

\maketitle

\begin{abstract}
We present a combinatorial model of configuration spaces and polytopes associated to the quotients of $\CC A_n$, the path algebra of the linearly oriented $A_n$ quiver, i.e. the algebra of upper triangular matrices. These quotient algebras are known as linear Nakayama algebras. Such configuration spaces were recently introduced for more general algebras by the second author and collaborators. In this special setting, we provide elementary proofs and explicit combinatorial constructions. From a Dyck path we define three related objects: a finite-dimensional algebra, an affine algebraic variety, and a polytope. Moreover, our constructions are natural: each relation in the poset of Dyck paths gives a morphism between the corresponding objects.
\end{abstract}

\section{Introduction}
The $n$-dimensional associahedron ${\rm Assoc}(n)$ is an important object in combinatorics and topology \cite{stasheff} and in the subject of positive geometry \cite{abhy}. One model for the face lattice of this polytope is to label each of its facets by a chord in a $n+3$-gon; the lower dimensional faces are indexed by subsets of pairwise noncrossing chords. This model of chords in a polygon can also be used to define ``dihedral coordinates'' (see \cite{brown}) for another important object: the moduli space of $n+3$ marked points in $\PP^1$, $\mathcal{M}_{0,n+3}$. Indeed, we can embed $\iota: \mathcal{M}_{0,n+3}\to (\CC^*)^{n(n+3)/2}$ into an algebraic torus whose coordinates are $u_{ij}$ where $ij$ is a chord in our $n+3$-gon. Each $u_{ij}$ is a particular cross-ratio of points in $\PP^1$. The equations cutting out $U_n^\circ:= \iota(\mathcal{M}_{0,n+3})$ are the ``$u$-equations'' of \cite{brown, kobanielsen}, which are an instance of cross-ratio identities (see \cite{lam})
\begin{equation}\label{intro:u}
u_{ij} + \prod_{\substack{kl \\ kl \text{ crosses } ij}} u_{kl} = 1.
\end{equation}
The affine closure in $\CC^{n(n+3)/2}$ of $U_n^\circ$, denoted $U_n$, is an interesting affine variety. In particular, the intersection poset of the divisors $u_{ij}=0$ in the positive part $U_n \cap \RR_+^{n(n+3)/2}$ is precisely the face poset of ${\rm Assoc}(n)$. This follows by observing that $U_n$ is isomorphic to an open affine inside the projective toric variety of an appropriate realization of ${\rm Assoc}(n)$, see \cite{stringy}. 

\smallskip

The associahedron is now just the first of many known examples of this phenomenon, which extends beyond the classical cross-ratio identities. It was shown by \cite{clustconf} that a similar construction can be associated to generalized associahedra for finite type cluster algebras. Further, \cite{had} extends this to any \emph{finite representation type} algebra $\A$. For each of these algebras, $\A$, the affine algebraic variety $U_\A$, its configuration space, is defined by a system of polynomial equations of the form
\begin{equation}\label{intro:uu}
u_X + \prod_{Y \in \I_\A} u_Y^{c(X,Y)} = 1 ,
\qquad X \in \I_\A ,
\end{equation}
for a finite set of variables with an index set $\I_\A$, and for some \emph{compatibility degree} $c(X,Y) = c(Y,X) \in \ZZ_{\ge 0}$ determined by the algebra. In physics, these varieties define stringy integrals, which are integrals over the positive part $U^{\geq 0}_\A$ of these varieties. These are meromorphic functions whose properties follow from the properties of $U_\A$, as explained in \cite{stringy}.
\smallskip

In this paper we give a combinatorial model to understand the generalization of these configuration spaces for a family of algebras. We focus on quotient algebras of the path algebra of the linearly oriented $A_n$ quiver, also called \emph{linear Nakayama algebras}. Analogous to the associahedron case, we can fully describe the equations \eqref{intro:uu} pictorially.

\subsection{The linear quotient algebras}
We write $\mathbb{C}A_n$ for the \emph{path algebra} of the linearly oriented $A_n$ quiver. This is simply the algebra with a $\CC$-basis given by the set of paths $p_{ij}$, $1\leq i \leq j \leq n$, where $p_{ij}$ is the oriented path from node $j$ to node $i$. Multiplication is given by path concatenation, so that $p_{ij} p_{kl} = 0$ if $j \neq k$ and $p_{ij} p_{jk} = p_{ik}$. This multiplication rule is identical to matrix multiplication for the $n \times n$ unit matrices $E_{ij}$ ($i \leq j$), which form a basis of the upper triangular matrix algebra. Some readers may prefer to think of $\CC A_n$ as this matrix algebra.

\smallskip

The subject of the paper is the quotient algebras
\[
\A = \mathbb{C}A_n / I
\]
for two-sided ideals $I$ (see Section~\ref{sec:2A}). To each quotient algebra we give an explicit combinatorial definition of the ``configuration space'' $U_\A$, which is an affine variety of dimension $n$, and an $n$-dimensional polytope $P_{\A}$. Our main results concern the properties of $U_\A$ and $P_\A$. The starting point is the combinatorics of \emph{Dyck path} diagrams. 

\smallskip

In fact, the two-sided ideals $I \unlhd \mathbb{C}A_n$, and hence the quotient algebras $\A$, are in bijection with the set of Dyck paths with $2n$ steps (Proposition \ref{prop:dyck}). This combinatorial bijection follows naturally from a simple picture for the modules of these algebras. Consider the linearly oriented $A_3$ quiver, which we orient from right to left,
\[
1 \xlongleftarrow[]{~\alpha_1~} 2  \xlongleftarrow[]{~\alpha_2~}  3.
\]
The indecomposable $\mathbb{C}A_3$-modules can be labeled by their support, for example $M_{12} = \CC \leftarrow \CC \leftarrow 0$, and we arrange the modules in a triangular grid (the Auslander-Reiten quiver):
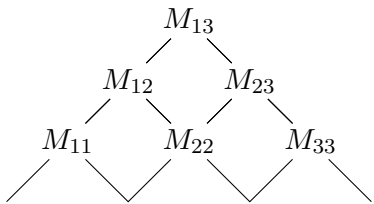
\[
\begin{tikzpicture}[every node/.style={inner sep=2pt},scale=0.8]
  \node (M11) at (-2,0) {$M_{11}$};
  \node (M22) at ( 0,0) {$M_{22}$};
  \node (M33) at ( 2,0) {$M_{33}$};
  \node (M12) at (-1,1) {$M_{12}$};
  \node (M23) at ( 1,1) {$M_{23}$};
  \node (M13) at (0,2) {$M_{13}$};
    \draw[black] (M12) -- (M22) -- (-1,-1) -- (M11) --  (M12);
  \draw[black] (M22) -- (M23) -- (M33) -- (1,-1) --  (M22);
  \draw[black] (M12) -- (M13) -- (M23) -- (M22) -- (M12);
  \draw[black] (-3,-1) -- (M11) -- (M12) -- (M13) -- (M23) -- (M33) -- (3,-1);
\end{tikzpicture}
\]
A choice of ideal $I\unlhd \mathbb{C}A_3$ is encoded by a Dyck path drawn in the same grid. For instance, the Dyck path in red below
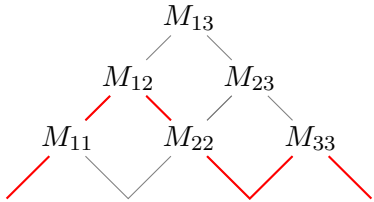
\[
\begin{tikzpicture}[every node/.style={inner sep=2pt},scale=0.8]
  \node (M11) at (-2,0) {$M_{11}$};
  \node (M22) at ( 0,0) {$M_{22}$};
  \node (M33) at ( 2,0) {$M_{33}$};
  \node (M12) at (-1,1) {$M_{12}$};
  \node (M23) at ( 1,1) {$M_{23}$};
  \node (M13) at (0,2) {$M_{13}$};
    \draw[gray] (M12) -- (M22) -- (-1,-1) -- (M11) --  (M12);
  \draw[gray] (M22) -- (M23) -- (M33) -- (1,-1) --  (M22);
  \draw[gray] (M12) -- (M13) -- (M23) -- (M22) -- (M12);
  \draw[thick, red] (-3,-1) -- (M11) -- (M12) -- (M22) -- (1,-1) -- (M33) -- (3,-1);
\end{tikzpicture}
\]
corresponds to the ideal $I=\langle \alpha_2\rangle$: the modules $M_{13}$ and $M_{23}$, which lie above the Dyck path, are precisely those which are not indecomposable modules of the quotient algebra $\A = \CC A_3 / I$.

\medskip

\subsection{Summary of results}
Fix a Dyck path $D$, and its corresponding quotient algebra $\A = \CC A_n / I$. Starting from the Dyck path diagram we define:
\begin{enumerate}
\item \emph{a polytope} $P_{\A}$ (equation \eqref{eq: polytope}),
\item \emph{an index set} $\I_{\A}$ (Definition~\ref{dfn:IADyck}), together with
\item \emph{a notion of compatibility and incompatibility} for pairs $X,Y \in \I_{\A}$ (Definition~\ref{dfn: comp degree}).
\end{enumerate}
We emphasize that these definitions have a combinatorial flavor. For example, $\I_\A$ is the set of steps of the Dyck path $D$, together with the squares that lie beneath it in the grid picture.
\smallskip

The index set $\I_\A$ labels the affine coordinates $u_X$ ($X \in \I_\A$) on the affine space where we define the configuration space $U_\A$. The equations that cut out $U_\A$ are given by the compatibility rule on pairs $X,Y \in \I_\A$ (Definition~\ref{dfn: u variety}). We show that the stratification of the nonnegative part $(U_\A)_{\geq 0}$ is given by pairwise compatible subsets of $\I_\A$, ordered by inclusion (Proposition~\ref{prop:fact}).
\smallskip

The polytope $P_\A$ is defined as a Minkowski sum of simplices associated to the points in the grid that lie \emph{on or below} the Dyck path $D$ (see Example~\ref{ex: intro} below). It can be helpful to view $P_\A$ as the moment polytope of its toric variety, $X_{P_\A}$, and we show that $U_\A$ is an affine open in this toric variety (Theorem~\ref{thm:toric}). In particular, we show that the face lattice of $P_\A$ is the same as the poset of strata of $(U_\A)_{\geq 0}$, by identifying $(U_\A)_{\geq 0}$ with the positive part of $X_{P_\A}$ (Corollary~\ref{cor:toric}). However, we study the polytopes $P_\A$ in a purely combinatorial way in Section \ref{sec:poly}. They are not generalized permutohedra in the sense of Postnikov \cite{sumsimplP}, but we can study them in a similar fashion. In particular, we show that $P_\A$ is a simple polytope (Proposition~\ref{prop:simple}).


\medskip

\begin{ex}\label{ex: intro}
We illustrate these definitions for the linearly oriented $A_2$ quiver
\[
1 \xlongleftarrow[]{~\alpha~} 2.
\]
There are two Dyck paths of length $2$, $D$ (a peak) and $\widetilde{D}$ (a valley). $D$ corresponds to the trivial ideal $I=\emptyset$ and the full path algebra $\A = \CC A_2$. Whereas $\widetilde{D}$ corresponds to the ideal $\widetilde{I}=\langle \alpha\rangle$ and quotient algebra ${\widetilde \A}=\CC A_2/\langle \alpha\rangle$. We draw these Dyck paths in the following grid pictures:
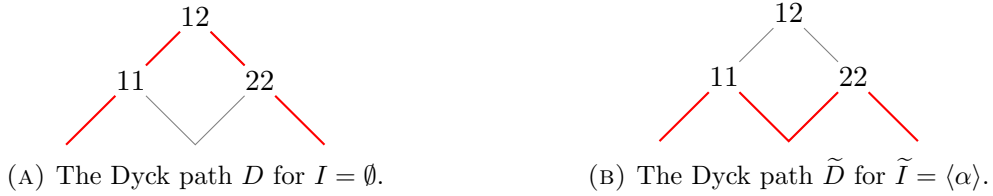
\begin{figure}[H]
\centering
\begin{subfigure}{0.48\textwidth}
\centering
\begin{tikzpicture}[scale=0.85,every node/.style={inner sep=2pt}]
  \node (M11) at (-1,0) {$11$};
  \node (M22) at ( 1,0) {$22$};
  \node (M12) at (0,1) {$12$};
  \draw[gray] (M11)--(M12)--(M22)--(0,-1)--(M11);
  \draw[thick, red] (-2,-1) -- (M11) -- (M12) -- (M22) -- (2,-1);
\end{tikzpicture}
\caption{The Dyck path $D$ for $I = \emptyset$.}
\end{subfigure}
\begin{subfigure}{0.48\textwidth}
\centering
\begin{tikzpicture}[scale=0.85,every node/.style={inner sep=2pt}]
  \node (M11) at (-1,0) {$11$};
  \node (M22) at ( 1,0) {$22$};
  \node (M12) at (0,1) {$12$};
  \draw[gray] (M11)--(M12)--(M22)--(0,-1)--(M11);
  \draw[thick, red] (-2,-1) -- (M11) -- (0,-1) -- (M22) -- (2,-1);
\end{tikzpicture}
\caption{The Dyck path $\widetilde{D}$ for $\widetilde{I}=\langle\alpha\rangle$.}
\end{subfigure}
\end{figure}

\noindent
The associated moment polytopes $P_{\A}$ and $P_{\widetilde{\A}}$ can be defined as a Minkowski sum, using these pictures. Each node of the grid diagram is labelled by an indecomposable $\A$-module, $M_{11}$, $M_{22}$, and $M_{12}$. To each of these, we associate a polynomial,
\[
F_{11}=1+y_1,\qquad F_{22}=1+y_2,\qquad F_{12}=1+y_1+y_1y_2,
\]
called \emph{$F$-polynomials} (see e.g. \cite{had}).
The polytopes $P_\A$ and $P_{\widetilde \A}$ are given by taking the Minkowski sum of $\Newt(F_{ij})$ for each $M_{ij}$ that lies \emph{on or below} the corresponding Dyck path:
\[
P_{\A}=\Newt(F_{11})+\Newt(F_{22})+\Newt(F_{12}),\qquad P_{\widetilde{\A}}=\Newt(F_{11})+\Newt(F_{22}).
\]
This gives a pentagon and a square, respectively:
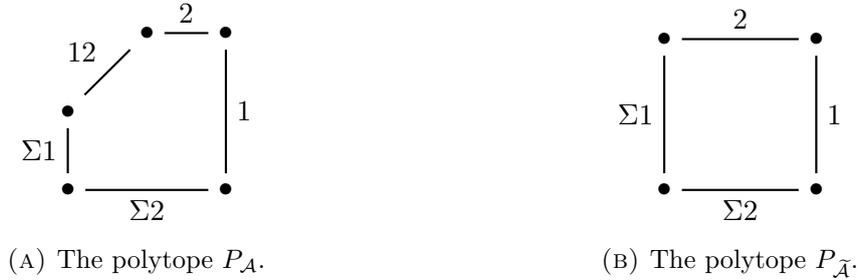
\begin{figure}[H]
\begin{subfigure}{0.48\textwidth}
\centering
\begin{tikzpicture}[scale=0.8]
  \node (A) at (0,0) {$\bullet$};
  \node (B) at (0,1.3) {$\bullet$};
  \node (C) at (1.3,2.6) {$\bullet$};
  \node (D) at (2.6,2.6) {$\bullet$};
  \node (E) at (2.6,0) {$\bullet$};
  \draw[thick] (A) --node[midway,left]{$\Sigma 1$} (B) --node[midway,above left]{$12$} (C) --node[midway, above]{$2$} (D) --node[midway, right]{$1$} (E) --node[midway,below]{$\Sigma 2$} (A);
\end{tikzpicture}
\caption{The polytope $P_{\A}$.}
\end{subfigure}
\begin{subfigure}{0.48\textwidth}
\centering
\begin{tikzpicture}
  \node (A) at (0,0) {$\bullet$};
  \node (B) at (2,0) {$\bullet$};
  \node (C) at (2,2) {$\bullet$};
  \node (D) at (0,2) {$\bullet$};
  \draw[thick] (A) --node[midway,below]{$\Sigma 2$}  (B) --node[midway, right]{$1$} (C) --node[midway, above]{$2$} (D) --node[midway,left]{$\Sigma 1$} (A);
\end{tikzpicture}
\caption{The polytope $P_{\widetilde \A}$.}
\end{subfigure}
\end{figure}

\noindent
The faces of each polytope are labelled by the set of steps of the corresponding Dyck path, together with the diamonds that lie below the Dyck path. We write $1,2,\ldots$ for the up steps of the path and $\Sigma 1, \Sigma 2,\ldots$ for the down steps, and we label the diamond below $M_{12}$ as $12$. Then the faces of the polytopes $P_\A$ and $P_{\widetilde \A}$ are labelled by (Definition~\ref{dfn:IADyck})
\[
\I_\A = \{ 1,2,\Sigma 1, \Sigma 2, 12\},\qquad \I_{\widetilde \A} = \{1,2,\Sigma 1, \Sigma 2\}.
\]
The varieties $U_\A$ and $U_{\widetilde \A}$ are given by affine equations in $\CC^{\I_\A}$ and $\CC^{\I_{\widetilde \A}}$, respectively. For $U_\A$, we introduce variables $u_X$ ($X \in \I_\A$), and the equations are
\[
u_{12} + u_1u_{\Sigma 2}=1, \qquad u_1 + u_{12}u_{\Sigma 1}=1,\qquad u_2 + u_{\Sigma 1}u_{\Sigma 2}=1, \qquad u_{\Sigma 1}+u_1u_2=1,\qquad u_{\Sigma 2}+u_2u_{12}=1.
\]
For $\widetilde{\A}$, we have variables $\widetilde{u}_X$ ($X \in \I_{\widetilde \A}$) and $U_{\widetilde \A}$ is cut out by the equations
\[
\widetilde{u}_1+ \widetilde{u}_{\Sigma 1}=1,\qquad \widetilde{u}_2+\widetilde{u}_{\Sigma 2}=1.
\]
$U_\A$ and $U_{\widetilde A}$ are 2-dimensional varieties. The positive part of $U_\A$ has five boundaries and the positive part of $U_{\widetilde A}$ has four boundaries, i.e. they are combinatorially a pentagon and a square.
\end{ex}

\bigskip

The Dyck paths of fixed length form a natural poset (i.e. $D \leq D'$ if the path $D$ lies on or below $D'$). For an ordered pair of Dyck paths $\widetilde D\le D$ their corresponding ideals are ordered by inclusion $I_{\widetilde D}\supseteq I_D$. In Theorem~\ref{thm:surj}, we show how this poset structure is also realized among the configuration spaces $U_\A$. For any relation $\widetilde{D} \leq D$, we define a monomial map $\phi : \CC[U_{\widetilde \A}] \rightarrow \CC[U_\A]$ of coordinate rings which induces a surjection $\phi^* : U_\A \rightarrow U_{\widetilde \A}$ (Definition~\ref{dfn: monomial maps}). The monomial map $\phi$ can be defined from the Dyck path pictures: it is given by looking at the diamonds that lie between $\widetilde{D}$ and $D$. We show that the maps $\phi$ realize the poset structure faithfully.

\begin{ex}
In the $A_2$ example above, the varieties $U_\A$ and $U_{\widetilde \A}$ are related by the map
\[
\phi: \widetilde{u}_1 \mapsto u_1,\qquad \widetilde{u}_2 \mapsto u_2 u_{12},\qquad \widetilde{u}_{\Sigma 1} \mapsto u_{12} u_{\Sigma 1},\qquad \widetilde{u}_{\Sigma 2} \mapsto u_{\Sigma 2},
\]
and one checks that the pullback $\phi^*$ defines a surjection onto $U_{\widetilde \A}$.
\end{ex}
\medskip

The toric viewpoint mentioned above provides an alternative interpretation of these monomial maps. Recall, we identify $U_{\A}$ as an affine open in a projective toric variety with moment polytope $P_{\A}$ (Theorem~\ref{thm:toric}). Under this identification, the monomial maps associated to relations in the Dyck path poset can be interpreted as toric blow-down maps (Proposition~\ref{prop: toric}).

\subsection{Outline}
Section~\ref{sec:2} establishes the correspondence between Dyck paths and quotient algebras of type $A$. Section~\ref{sec:3} defines the configuration spaces $U_\A$ for each Dyck path, and studies the monomial maps between them. In this section, Proposition~\ref{prop:fact} characterizes the coordinate divisors of $U_\A$. Section~\ref{sec:F} constructs the $F$-polynomial parametrization of $U_{\A}$ (Theorem~\ref{thm:param}) and studies the toric embedding. Section~\ref{sec:poly} gives a complete combinatorial description of the polytopes $P_\A$ and their face lattices and in particular shows that they are simple.
\medskip

\paragraph{Acknowledgements.}
We thank Hugh Thomas, Bernd Sturmfels, Dhruv Sidana, and Sagnik Das for helpful comments. HF is supported by the Sivian Fund. Both authors received additional support from the European Union (ERC, UNIVERSE PLUS, 101118787).\footnote{\tiny Views and opinions expressed are however those of the author(s) only and do not necessarily reflect those of the European Union or the European Research Council Executive Agency. Neither the European Union nor the granting authority can be held responsible for them.}

\section{Quotient Algebras and Dyck Paths}\label{sec:2}
In this section we set up the combinatorial model for our paper. We start by recalling some quiver representation theory restricted to our setting (we refer to \cite{assoc,QuiverBook} for a detailed exposition).

\subsection{Type $A$ path algebras}\label{sec:2A}
Our starting point is the linearly oriented $A_n$ quiver:
\[
   1 \longleftarrow 2  \longleftarrow 3 \longleftarrow \cdots \longleftarrow n-1 \longleftarrow n.
\]
We now consider paths in this quiver: we denote $e_1,\ldots, e_n$ for the trivial paths where we remain at a vertex and $\alpha_i, 1\leq i \leq n-1$ for the directed edges $i+1 \to i$. Then, all non trivial paths are concatenations of compatible edges $\alpha_j\alpha_{j-1}\cdots \alpha_i$ for pairs $1\leq i \leq j \leq n-1$. The \emph{path algebra} of $A_n$, denoted $\CC A_n$, is the algebra with basis all paths of $A_n$ and multiplication is concatenation.

\smallskip

A module $M$ for $\CC A_n$ is equivalently a \emph{quiver representation} of $A_n$. A quiver representation of $A_n$ (with linear orientation as above) is a vector space $V_i$ for each vertex and a collection of linear maps, $f_i: V_{i+1} \to V_{i}$, for each arrow. $M$ is called \emph{indecomposable} if it cannot be written as a direct sum of non-trivial modules. We can explicitly describe all indecomposable modules for $\CC A_n$ up to isomorphism. They are given by choosing a subset of adjacent vertices in $A_n$ to which we assign $V_i = \CC$, the one dimensional vector space, while all other vertices are zero. The maps $f_i$ are given by the identity if $V_i, V_{i+1}$ are both non-zero. We label these modules $M_{i,j}$ for $1\leq i \leq j \leq n$. As a quiver representation $M_{i,j}$ is
\[
    0 \xlongleftarrow[]{0} \cdots \xlongleftarrow[]{0} \underbrace{\CC}_i \xlongleftarrow[]{\text{id}} \CC \xlongleftarrow[]{\text{id}} \cdots \xlongleftarrow[]{\text{id}} \underbrace{\CC}_j \xlongleftarrow[]{0} 0 \xlongleftarrow[]{0} \cdots \xlongleftarrow[]{0} 0.
\]
We will draw diagrams that place these modules in a grid where each module $M_{i,j}$ is at location $(i,j)$. For convenience, we draw this grid as a triangle above the diagonal. For example, the indecomposable modules of $A_3$ are
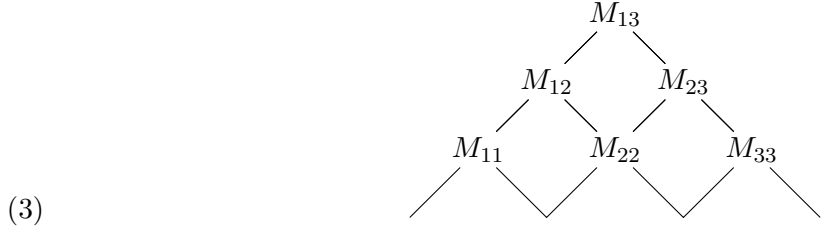
\begin{equation}\label{eq:grid3}
\begin{tikzpicture}[every node/.style={inner sep=2pt},scale=0.9]
  \node (M11) at (-2,0) {$M_{11}$};
  \node (M22) at ( 0,0) {$M_{22}$};
  \node (M33) at ( 2,0) {$M_{33}$};
  \node (M12) at (-1,1) {$M_{12}$};
  \node (M23) at ( 1,1) {$M_{23}$};
  \node (M13) at (0,2) {$M_{13}$};
    \draw[black] (M12) -- (M22) -- (-1,-1) -- (M11) --  (M12);
  \draw[black] (M22) -- (M23) -- (M33) -- (1,-1) --  (M22);
  \draw[black] (M12) -- (M13) -- (M23) -- (M22) -- (M12);
  \draw[black] (-3,-1) -- (M11) -- (M12) -- (M13) -- (M23) -- (M33) -- (3,-1);
\end{tikzpicture}
\end{equation}
Drawn this way, these grid diagrams agree with the conventions for drawing the \emph{Auslander-Reiten quiver} of the module category.

The family of algebras in which we focus, namely linear Nakayama algebras, are precisely the quotient algebras of $\CC A_n$. For this we consider two-sided ideals $I \unlhd \CC A_n$. Let $\A$ be such a quotient algebra $\CC A_n/I$. $\A$-modules correspond in quiver language to representations of \emph{quivers with relations}, which can be denoted by the pair $(A_n,I)$. (See \cite[Chapters II and III]{assoc}.)

\begin{ex}\label{ex: A3}
Consider $1 \leftarrow 2 \leftarrow 3$. The path algebra $\CC A_3$ has six indecomposable modules,
\begin{align*}
	M_{1,1} && \CC \leftarrow 0 \leftarrow 0 && ~ && M_{2,2} && 0 \leftarrow \CC \leftarrow 0\\
	M_{1,2} && \CC \leftarrow \CC \leftarrow 0 && ~ && M_{2,3} && 0 \leftarrow \CC \leftarrow \CC\\
	M_{1,3} && \CC\leftarrow \CC \leftarrow \CC && ~ && M_{3,3} && 0 \leftarrow 0 \leftarrow \CC
\end{align*}
$\CC A_3$ has five two-sided ideals without idempotents which naturally form a poset, by inclusion:
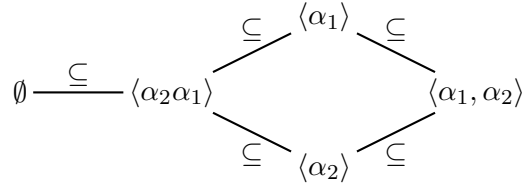
\[
\begin{tikzpicture}[
  every node/.style={inner sep=2pt},
  >=stealth,
  ed/.style={thick}, lab/.style={midway, inner sep=2pt}
]
  \node (a) at (0,0) {$\emptyset$};
  \node (b) at (2,0) {$\langle \alpha_2 \alpha_1 \rangle$};
  \node (c) at (4,1) {$\langle \alpha_1\rangle$};
  \node (d) at (4,-1) {$\langle \alpha_2 \rangle$};
  \node (e) at (6,0) {$\langle \alpha_1, \alpha_2 \rangle$};
\draw[ed] (a) -- (b) node[lab, above] {$\subseteq$};
  \draw[ed] (b) -- (c) node[lab, above=2pt] {$\subseteq$};
  \draw[ed] (b) -- (d) node[lab, below=2pt] {$\subseteq$};
  \draw[ed] (c) -- (e) node[lab, above=2pt] {$\subseteq$};
  \draw[ed] (d) -- (e) node[lab, below=2pt] {$\subseteq$};
\end{tikzpicture}
\]
The indecomposable modules of each of the quotient algebras $\CC A_3/I$ are identified with subsets of the six indecomposable modules of $\CC A_3$. Take $I =\langle \alpha_2\alpha_1 \rangle$ and $\A = \CC A_3/I$. The module $M_{1,3}$ does not satisfy the relations of $I$: the composition, in $M_{1,3}$, of the two linear maps corresponding to $\alpha_1$ and $\alpha_2$, is the identity map (i.e. not zero). All other $M_{i,j}$ modules satisfy the relations and hence correspond to $\A$-modules.
\end{ex}

\subsection{Dyck paths and combinatorial model}
We now give a combinatorial way to build all the quotient algebras $\mathbb{C}A_n/I$ and identify their indecomposable modules by using \emph{Dyck paths}. For us, a \emph{Dyck path of length $n$} is a lattice walk in our grid picture for the indecomposable modules of $\mathbb{C}A_n$ (e.g. see \eqref{eq:grid3} for $n=3$). Each Dyck path begins at position $(1,0)$ on the left, and ends at $(n+1,n)$ on the right, and remains above the baseline. Each step is either an up step, $(i,j) \rightarrow (i,j+1)$, or a down step, $(i,j) \rightarrow (i+1,j)$. For example, we can draw the following Dyck path (shown in red) for $n=3$:
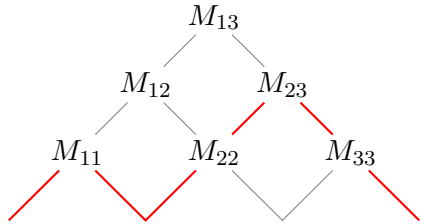
\[
\begin{tikzpicture}[every node/.style={inner sep=2pt},scale=0.9]
  \node (M11) at (-2,0) {$M_{11}$};
  \node (M22) at ( 0,0) {$M_{22}$};
  \node (M33) at ( 2,0) {$M_{33}$};
  \node (M12) at (-1,1) {$M_{12}$};
  \node (M23) at ( 1,1) {$M_{23}$};
  \node (M13) at (0,2) {$M_{13}$};
  \draw[thick, red] (-3,-1) -- (M11) -- (-1,-1) -- (M22) -- (M23) -- (M33) -- (3,-1);
  \draw[gray] (M11) -- (M12) -- (M22) -- (1,-1) -- (M33);
  \draw[gray] (M12) -- (M13) -- (M23);
\end{tikzpicture}
\]

\newcommand{\picEmpty}{%
  \begin{tikzpicture}[baseline=-0.6ex,scale=0.4]
  \coordinate (M11) at (-2,0);
  \coordinate (M22) at ( 0,0);
  \coordinate (M33) at ( 2,0);
  \coordinate (M12) at (-1,1);
  \coordinate (M23) at ( 1,1);
  \coordinate (M13) at (0,2);
  \coordinate (P1) at (-3,-1);
  \coordinate (P2) at (-1,-1);
  \coordinate (P3) at (1,-1) ;
  \coordinate (P4) at (3,-1) ;
    \draw[gray] (M12) -- (M22) -- (P2) -- (M11) --  (M12);
  \draw[gray] (M22) -- (M23) -- (M33) -- (P3) --  (M22);
  \draw[gray] (M12) -- (M13) -- (M23) -- (M22) -- (M12);
  \draw[thick, red] (P1) -- (M11) -- (M12) -- (M13) -- (M23) -- (M33) -- (P4);
  \end{tikzpicture}%
}
\newcommand{\picTwoOne}{%
  \begin{tikzpicture}[baseline=-0.6ex,scale=0.4]
  \coordinate (M11) at (-2,0);
  \coordinate (M22) at ( 0,0);
  \coordinate (M33) at ( 2,0);
  \coordinate (M12) at (-1,1);
  \coordinate (M23) at ( 1,1);
  \coordinate (M13) at (0,2);
  \coordinate (P1) at (-3,-1);
  \coordinate (P2) at (-1,-1);
  \coordinate (P3) at (1,-1) ;
  \coordinate (P4) at (3,-1) ;
    \draw[gray] (M12) -- (M22) -- (P2) -- (M11) --  (M12);
  \draw[gray] (M22) -- (M23) -- (M33) -- (P3) --  (M22);
  \draw[gray] (M12) -- (M13) -- (M23) -- (M22) -- (M12);
  \draw[thick, red] (P1) -- (M11) -- (M12) -- (M22) -- (M23) -- (M33) -- (P4);
  \end{tikzpicture}%
}

\newcommand{\picOne}{%
  \begin{tikzpicture}[baseline=-0.6ex,scale=0.4]
  \coordinate (M11) at (-2,0);
  \coordinate (M22) at ( 0,0);
  \coordinate (M33) at ( 2,0);
  \coordinate (M12) at (-1,1);
  \coordinate (M23) at ( 1,1);
  \coordinate (M13) at (0,2);
  \coordinate (P1) at (-3,-1);
  \coordinate (P2) at (-1,-1);
  \coordinate (P3) at (1,-1) ;
  \coordinate (P4) at (3,-1) ;
    \draw[gray] (M12) -- (M22) -- (P2) -- (M11) --  (M12);
  \draw[gray] (M22) -- (M23) -- (M33) -- (P3) --  (M22);
  \draw[gray] (M12) -- (M13) -- (M23) -- (M22) -- (M12);
  \draw[thick, red] (P1) -- (M11) -- (P2) -- (M22) -- (M23) -- (M33) -- (P4);
  \end{tikzpicture}%
}

\newcommand{\picTwo}{%
  \begin{tikzpicture}[baseline=-0.6ex,scale=0.4]
  \coordinate (M11) at (-2,0);
  \coordinate (M22) at ( 0,0);
  \coordinate (M33) at ( 2,0);
  \coordinate (M12) at (-1,1);
  \coordinate (M23) at ( 1,1);
  \coordinate (M13) at (0,2);
  \coordinate (P1) at (-3,-1);
  \coordinate (P2) at (-1,-1);
  \coordinate (P3) at (1,-1) ;
  \coordinate (P4) at (3,-1) ;
    \draw[gray] (M12) -- (M22) -- (P2) -- (M11) --  (M12);
  \draw[gray] (M22) -- (M23) -- (M33) -- (P3) --  (M22);
  \draw[gray] (M12) -- (M13) -- (M23) -- (M22) -- (M12);
  \draw[thick, red] (P1) -- (M11) -- (M12) -- (M22) -- (P3) -- (M33) -- (P4);
  \end{tikzpicture}%
}

\newcommand{\picBottom}{%
  \begin{tikzpicture}[baseline=-0.6ex,scale=0.4]
  \coordinate (M11) at (-2,0);
  \coordinate (M22) at ( 0,0);
  \coordinate (M33) at ( 2,0);
  \coordinate (M12) at (-1,1);
  \coordinate (M23) at ( 1,1);
  \coordinate (M13) at (0,2);
  \coordinate (P1) at (-3,-1);
  \coordinate (P2) at (-1,-1);
  \coordinate (P3) at (1,-1) ;
  \coordinate (P4) at (3,-1) ;
    \draw[gray] (M12) -- (M22) -- (P2) -- (M11) --  (M12);
  \draw[gray] (M22) -- (M23) -- (M33) -- (P3) --  (M22);
  \draw[gray] (M12) -- (M13) -- (M23) -- (M22) -- (M12);
  \draw[thick, red] (P1) -- (M11) -- (P2) -- (M22) -- (P3) -- (M33) -- (P4);
  \end{tikzpicture}%
}

\begin{figure}
\centering
\begin{tikzpicture}[
  every node/.style={inner sep=2pt},
  >=stealth,
  ed/.style={thick}, lab/.style={midway, inner sep=2pt}
]
  \node (a) at (-2,0) {\picEmpty};
  \node (b) at (2,0) {\picTwoOne};
  \node (c) at (6,1) {\picOne};
  \node (d) at (6,-1) {\picTwo};
  \node (e) at (10,0) {\picBottom};
  \draw[ed] (a) -- (b) node[lab, above=1pt] {$\geq$};
  \draw[ed] (b) -- (c) node[lab, above=2pt] {$\geq$};
  \draw[ed] (b) -- (d) node[lab, below=2pt] {$\geq$};
  \draw[ed] (c) -- (e) node[lab, above=2pt] {$\geq$};
  \draw[ed] (d) -- (e) node[lab, below=2pt] {$\geq$};
\end{tikzpicture}
\caption{Dyck path poset for $n=3$.}
\label{fig:DyckPoset3}
\end{figure}
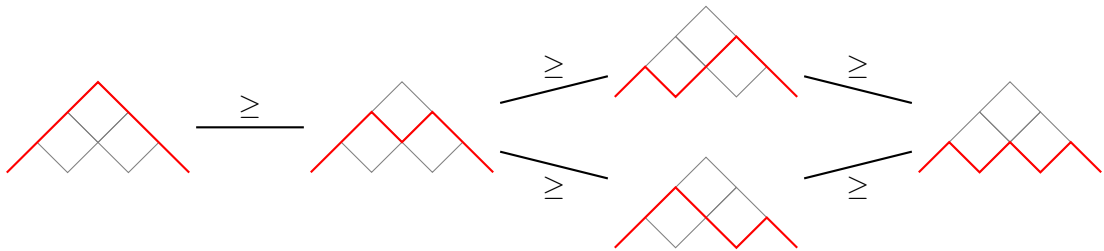

Write $\D_n$ for the collection of all Dyck paths of length $n$. $\D_n$ is a poset, with partial order given by $D \leq D'$ if and only if $D$ lies on/below $D'$. For example, $\D_3$ is the poset in Figure~\ref{fig:DyckPoset3}. The similarity of Figure \ref{fig:DyckPoset3} with the poset in Example~\ref{ex: A3} is not a coincidence. In fact, to each Dyck path $D \in \D_n$ we can associate an ideal $I_D$ in $\CC A_n$, which we define as follows. The generators of $I_D$ correspond to the valleys of $D$, those points that are preceded by a down step and followed by an up step. If $(a,b)$ is a valley of $D$, we consider the path from $b+1$ to $a-1$ in the $A_n$ quiver, which defines the element $p_{a,b}:=\alpha_{b}\cdots \alpha_{a-1}$ in the path algebra $\mathbb{C} A_n$. Then let $I_D$ be the two-sided ideal generated by these paths,
\begin{equation}\label{eq: ideal Dyck path}
    I_D := \langle p_{a,b} \, | \, (a,b) \text{ is a valley of } D \rangle.
\end{equation}

\begin{ex}
In the $n=3$ grid, there are precisely three lattice points where a Dyck path can have a valley. These are $(2,1)$, $(2,2)$, and $(3,2)$, which correspond to the paths $p_{2,1} = \alpha_1$, $p_{2,2} = \alpha_2 \alpha_1$, and $p_{3,2} = \alpha_2$, respectively. Given this, the ideals $I_D$ defined by each of the Dyck paths in Figure \ref{fig:DyckPoset3} correspond precisely to the poset of ideals of $\mathbb{C}A_3$ in Example~\ref{ex: A3}.
\end{ex}

We formally state the correspondence between ideals in $\CC A_n$ and Dyck paths in $\mathcal{D}_n$ as a poset isomorphism in the following Proposition.

\begin{prop}\label{prop:dyck}
The map $D \mapsto I_D$ is a poset isomorphism between $(\D_n,\geq)$ and two-sided ideals of $\CC A_n$ without idempotents ordered by inclusion. Moreover, the indecomposable modules of $\CC A_n/I_D$ are in one-to-one correspondence with the off-diagonal points on and below $D$.
\end{prop}

\begin{proof}
Let us prove the first part. Note that we can identify an ideal $I \subset \CC A_n$ (with no idempotents) with the collection of paths $\mathcal{P}_I$ in $A_n$ that belong to $I$ (this will in particular generate $I$). Consider the poset of paths $(\Pi, \preccurlyeq)$ with the natural order of subpaths. The collection of paths corresponding to ideals satisfy the following property:
\begin{center}
If $\alpha_i\cdots \alpha_j \in \mathcal{P}_I$ then any concatenation $\alpha_{i'}\cdots\alpha_{j'}$ with $i'\leq i$ and $j\leq j'$ is also in $\mathcal{P}_I$.
\end{center}
Hence, $\mathcal{P}_I$ are upsets of $(\Pi, \preccurlyeq)$. Furthermore, the poset of ideals of $\CC A_n$ is isomorphic to the poset of order ideals (using upsets instead of the usual downset convention) of $\Pi$. Note that we can associate to each path $p$ a tuple $(i,j)$ indicating the vertices where the path starts and finishes. The subpath relation $\preccurlyeq$ translates to $(i,j) \preccurlyeq (i',j')$ if and only if $i\geq i'$ and $j' \geq j$. The poset isomorphism then follows from \cite[Section 3.3]{dyckPaths} where it is proven that $(\D_n,\geq)$ is the order ideal poset of such tuples.

\smallskip

To prove the second part, note that the indecomposable modules of $\CC A_n/I_D$ are precisely the indecomposable modules of $\CC A_n$ that satisfy the relations in $I_D$. Let $p=\alpha_l\alpha_{l-1}\cdots \alpha_k \in \CC A_n$. The module $M_{i,j}$ satisfies the relation $p=0$ if and only if $i\geq k$ or $j\leq l+1$. Then, $M_{i,j}$ is a $\CC A_n/I_D$ indecomposable module if and only if for all valleys of $D$, say $(a,b)$, we have $i \geq a$ or $j\leq b$ which is equivalent to asking that $M_{i,j}$ is on or below $D$.
\end{proof}

In view of Proposition~\ref{prop:dyck}, given a Dyck path $D$, it is convenient to draw a grid diagram that includes only the points on or below $D$. These points then correspond to the irreducible modules of $\mathbb{C}A_n / I_D$. For example, taking the ideal $I = \langle \alpha_1 \rangle$, we draw the correspond Dyck path, $D$, which has one valley at $(2,1)$, and then the $4$ indecomposable modules of $\mathbb{C}A_3 / \langle \alpha_1 \rangle$ are
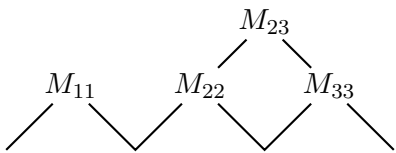
\[
\begin{tikzpicture}[every node/.style={inner sep=2pt},scale=0.85]
  \node (M11) at (-2,0) {$M_{11}$};
  \node (M22) at ( 0,0) {$M_{22}$};
  \node (M33) at ( 2,0) {$M_{33}$};
  \node (M23) at ( 1,1) {$M_{23}$};
  \draw[thick, black] (-3,-1) -- (M11) -- (-1,-1) -- (M22) -- (M23) -- (M33) -- (3,-1);
  \draw[thick, black] (M22) -- (1,-1) -- (M33);
\end{tikzpicture}
\]
We call such a diagram a \emph{truncated grid}.

\section{Monomial Maps and $u$ equations}\label{sec:3}
In this section we introduce the configuration space of a linear Nakayama algebra. Recall that in the previous section, we saw that the quotient algebras $\A = \CC A_n/I_D$ correspond one-to-one to Dyck paths $D$. We will use the combinatorics of Dyck paths to associate an affine variety $U_\A$ to each $\A$. Moreover, we show that the assignment $U_\A$ to $\A$ is functorial: each algebra quotient, $\A \rightarrow \B$, or equivalently each pair of ideals $I \subset J$, corresponds to a surjective map $U_\A \rightarrow U_{\B}$. These surjections realize the poset of ideals/Dyck paths discussed in the previous section.

\subsection{The affine varieties}
Fix a Dyck path $D$ and its associated quotient algebra $\A$ as in \eqref{eq: ideal Dyck path}. The varieties $U_\A$ we will consider naturally live in the affine space $\CC^{\I_\A}$ whose coordinates $u_X$ are indexed by the set $\I_\A$ given in Definition~\ref{dfn:IADyck}.
\begin{dfn}\label{dfn:IADyck}
Let $D$ be a Dyck path and $\A = \CC A_n/I_D$. We define an indexing set $\I_\A$, given as the set of steps of $D$ together with all the diamonds that lie below $D$ in the truncated grid.
\end{dfn}
\noindent
For convenience, we label, in order, the up steps of $D$ as $1,2,\ldots,n$, and the down steps of $D$ as $\Sigma 1, \Sigma 2, \ldots, \Sigma n$. The diamonds below $D$ are labelled by their the coordinates of their top corner. For example, Figure \ref{fig:IA5} shows the labels of $\I_\A$ for the truncated grid picture for $\A = \CC A_5 / \langle \alpha_2\alpha_1 \rangle$.

\begin{figure}[h!]
\centering
\begin{tikzpicture}[scale=1,every node/.style={inner sep=2pt}]
  \coordinate (M0) at (-3,-1);
  \coordinate (M1) at (-1,-1);
  \coordinate (M2) at (1,-1);
  \coordinate (M3) at (3,-1);
  \coordinate (M4) at (5,-1);
  \coordinate (M5) at (7,-1);
  \coordinate (M11) at (-2,0);
  \coordinate (M22) at ( 0,0);
  \coordinate (M33) at ( 2,0);
  \coordinate (M44) at ( 4,0);
  \coordinate (M55) at ( 6,0);
  \coordinate (M12) at (-1,1);
  \coordinate (M23) at ( 1,1);
  \coordinate (M34) at ( 3,1);
  \coordinate (M45) at ( 5,1);
  \coordinate (M24) at (2,2);
  \coordinate (M35) at (4,2);
  \coordinate (M25) at (3,3);
{\small
  \node (U11) at (-1,0) {${12}$};
  \node (U22) at ( 1,0) {${23}$};
  \node (U33) at ( 3,0) {${34}$};
  \node (U44) at ( 5,0) {${45}$};
  \node (U23) at (2,1) {${24}$};
  \node (U34) at (4,1) {${35}$};
  \node (U24) at (3,2) {${25}$};
}
  \draw[thick] (M11) -- (M1) -- (M22) -- (M2) -- (M33) -- (M3) -- (M44) -- (M4) -- (M55);
  \draw[thick] (M23) -- (M33) -- (M34) -- (M44) -- (M45);
  \draw[thick] (M24) -- (M34) -- (M35);
  {\small
  \draw[thick] (M0) --node[midway, above left]{$1$} (M11) --node[midway, above left]{$2$} (M12) --node[midway, above right]{$\Sigma 1$} (M22) --node[midway, above left]{$3$} (M23) --node[midway, above left]{$4$} (M24) -- node[midway, above left]{$5$} (M25) --node[midway, above right]{$\Sigma 2$} (M35) --node[midway, above right]{$\Sigma 3$} (M45) --node[midway, above right]{$\Sigma 4$} (M55) --node[midway, above right]{$\Sigma 5$} (M5);
  }
\end{tikzpicture}
\caption{The index set $\I_\A$ for $\A=\CC A_5/\langle \alpha_1\alpha_2\rangle$ is given by the steps of the Dyck path $D$ (labelled $1,2,3,4,5$ and $\Sigma 1, \Sigma 2, \Sigma 3, \Sigma 4, \Sigma 5$) and the diamonds beneath the Dyck path (labelled $ij$).}
\label{fig:IA5}
\end{figure}
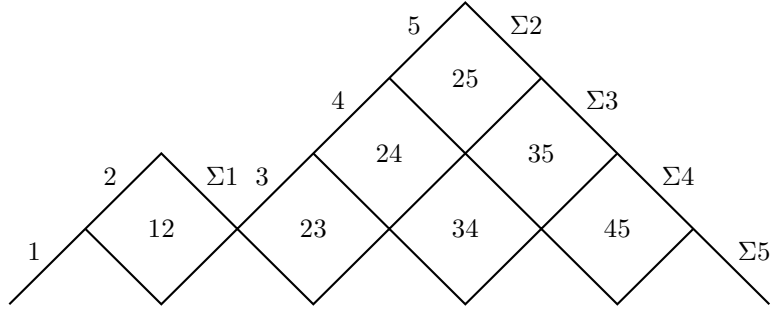

\begin{rmk}
In representation theory language, $\I_\A$ is the set of indecomposable objects in the category of 2-term projective complexes in $K^b(\text{proj} \A)$. Each up step of the Dyck path corresponds to a projective object $(0 \rightarrow P)$ in this category, where $P$ is a projective $\A$-module. Each down step corresponds to an injective object $(P \rightarrow 0)$. And each diamond $(i,j)$ corresponds to a 2-term projective presentation $(P \rightarrow P')$ of a non-projective $\A$-module. See \cite{had} for details.
\end{rmk}

\begin{dfn}\label{dfn: comp degree}
    Given two elements $X,Y \in \I_\A$ we say they are \emph{compatible} if the following~holds.
    \begin{itemize}
\item $X = (i,j)$ and $Y = (k,l)$ are compatible if the intervals $[i,j]$ and $[k,l]$ are
	\begin{itemize}
	\item comparable, $[i,j] \subset [k,l]$ or $[i,j] \supset [k,l]$, or
	\item do not intersect, $[i,j] \cap [k,l] = \emptyset$.
	\end{itemize}
\item $X = (i,j)$ and $Y = (k)$ are compatible if $k<i$ or $j\leq k$.
\item $X = (i,j)$ and $Y = (\Sigma k)$ are compatible if $k \leq i$ or $j<k$.
\item $X = (i)$ and $Y = (\Sigma j)$ are compatible if $i < j$ or the down step $\Sigma j$ comes before the up step $i$ in $D$. 
\end{itemize}
\end{dfn}

The rules for checking which pairs of elements are compatible or not have a simple interpretation in the truncated grid picture. To any diamond $(i,j)$ in the grid, we can draw the largest possible quadrilaterals to the left and right of the diamond. Any diamond in one of these quadrilaterals is incompatible with $(i,j)$, and any up-step in the left quadrilateral, or down-step in the right quadrilateral is also incompatible with $(i,j)$. Similarly, to any up step $(i)$ we draw the largest possible quadrilateral to the right, and to any down step $(\Sigma i)$ we draw the largest possible quadrilateral to the left. These \emph{light ray} pictures then determine all compatible pairs.
\begin{figure}[h!]
\centering
\begin{tikzpicture}[scale=1,every node/.style={inner sep=2pt}]
  \coordinate (M0) at (-3,-1);
  \coordinate (M1) at (-1,-1);
  \coordinate (M2) at (1,-1);
  \coordinate (M3) at (3,-1);
  \coordinate (M4) at (5,-1);
  \coordinate (M5) at (7,-1);
  \coordinate (M11) at (-2,0);
  \coordinate (M22) at ( 0,0);
  \coordinate (M33) at ( 2,0);
  \coordinate (M44) at ( 4,0);
  \coordinate (M55) at ( 6,0);
  \coordinate (M12) at (-1,1);
  \coordinate (M23) at ( 1,1);
  \coordinate (M34) at ( 3,1);
  \coordinate (M45) at ( 5,1);
  \coordinate (M24) at (2,2);
  \coordinate (M35) at (4,2);
  \coordinate (M25) at (3,3);
{\small
  \node[diamond, text=red, fill = red!20, minimum size= 2cm] (U11) at (-1,0) {${12}$};
  \node[blue] (U22) at ( 1,0) {${23}$};
  \node[blue] (U33) at ( 3,0) {${34}$};
  \node[diamond, text=red, fill = red!20, minimum size= 2cm] (U44) at ( 5,0) {${45}$};
  \node (U23) at (2,1) {$\mathbf{24}$};
  \node[diamond, text=red, fill = red!20, minimum size= 2cm] (U34) at (4,1) {${35}$};
  \node[blue] (U24) at (3,2) {${25}$};
}
  \draw[thick] (M11) -- (M1) -- (M22) -- (M2) -- (M33) -- (M3) -- (M44) -- (M4) -- (M55);
  \draw[thick] (M23) -- (M33) -- (M34) -- (M44) -- (M45);
  \draw[thick] (M24) -- (M34) -- (M35);
  {\small
  \draw[thick] (M0) --node[blue, midway, above left]{$1$} (M11) --node[red, midway, above left]{$2$} (M12) --node[blue, midway, above right]{$\Sigma 1$} (M22) --node[red, midway, above left]{$3$} (M23) --node[blue, midway, above left]{$4$} (M24) -- node[blue, midway, above left]{$5$} (M25) --node[blue, midway, above right]{$\Sigma 2$} (M35) --node[red, midway, above right]{$\Sigma 3$} (M45) --node[red, midway, above right]{$\Sigma 4$} (M55) --node[blue, midway, above right]{$\Sigma 5$} (M5);
  }
  \draw[thick, red] (M34) -- (M35);
  \draw[thick, red] (M34) -- (M44) -- (M4) -- (M55);
  \draw[thick, red] (M23) -- (M22) -- (M1) -- (M11);
\end{tikzpicture}
\caption{For $\A=\CC A_5/\langle \alpha_1\alpha_2\rangle$ and $X = 24 \in \I_\A$, the compatible $Y \in \I_\A$ (for which $c(X,Y) = 0$) are shown in blue, and the incompatible are shown in red.}
\label{fig:urel24}
\end{figure}
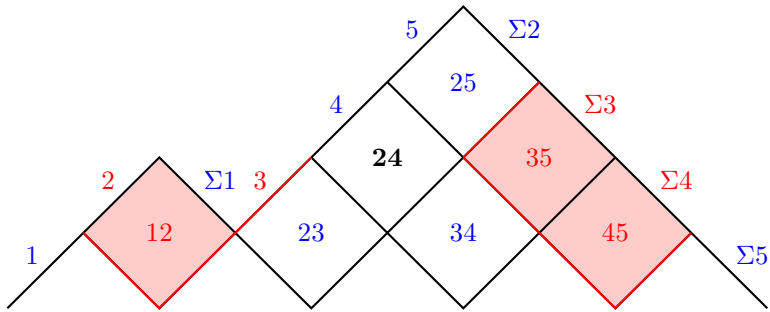

We now use the compatibility of elements in $\I_\A$ to define the configuration space of $\A$ as the solutions to $|\I_\A|$ many equations. It will be useful for us to consider both an affine variety $U_\A$ and a very affine variety $U_\A^\circ$ with the same set of generators.

\begin{dfn}\label{dfn: u variety}
    Let $\A$ be a linear Nakayama algebra and $\I_\A$ be as in Definition~\ref{dfn:IADyck}. The \emph{configuration space of $\A$} is the affine variety $U_\A \subset \CC^{\I_\A}$ given by the \emph{$u$-equations}:
    \begin{equation}\label{eq: ueqs}
        u_X + \prod_{Y \text{ incompatible with } X} u_Y = 1,\qquad  \text{for } X \in \I_{\A}.
    \end{equation}
    We write $U_\A^\circ \subset (\CC^*)^{\I_\A}$ for the very affine variety defined by the same set of equations.
\end{dfn}

We will see in Section~\ref{sec:3} that, even if defined by a square system of equations, the varieties $U_\A$ and $U_\A^\circ$ are irreducible and have dimension $n$. Moreover, their vanishing ideal is precisely generated by the $u$-equations.

\begin{ex}\label{ex: A5 urels}
Consider $\A = \CC A_5 / \langle \alpha_2\alpha_1 \rangle$. The Dyck path and truncated grid for this algebra are shown in Figures \ref{fig:urel24} and \ref{fig:urel2}. Figure \ref{fig:urel24} shows the $Y \in \I_\A$ incompatible with $X = (2,4)$. The two large quadrilaterals to the left and right of $(2,4)$ are shown in red. The associated $u$-equation is
\[
u_{24} + u_2 u_3 u_{12} u_{35} u_{45} u_{\Sigma 3} u_{\Sigma 4} = 1.
\]
Figure \ref{fig:urel2} shows elements of $\I_\A$ incompatible with $X = (2)$, with associated $u$-equation
\[
u_2 + u_{23} u_{24} u_{25} u_{\Sigma 1} u_{\Sigma 2} = 1.
\]
\end{ex}

\begin{figure}
\centering
\begin{tikzpicture}[scale=1,every node/.style={inner sep=2pt}]
  \coordinate (M0) at (-3,-1);
  \coordinate (M1) at (-1,-1);
  \coordinate (M2) at (1,-1);
  \coordinate (M3) at (3,-1);
  \coordinate (M4) at (5,-1);
  \coordinate (M5) at (7,-1);
  \coordinate (M11) at (-2,0);
  \coordinate (M22) at ( 0,0);
  \coordinate (M33) at ( 2,0);
  \coordinate (M44) at ( 4,0);
  \coordinate (M55) at ( 6,0);
  \coordinate (M12) at (-1,1);
  \coordinate (M23) at ( 1,1);
  \coordinate (M34) at ( 3,1);
  \coordinate (M45) at ( 5,1);
  \coordinate (M24) at (2,2);
  \coordinate (M35) at (4,2);
  \coordinate (M25) at (3,3);
{\small
  \node[blue] (U11) at (-1,0) {${12}$};
  \node[diamond, text=red, fill = red!20, minimum size= 2cm] (U22) at ( 1,0) {${23}$};
  \node[blue] (U33) at ( 3,0) {${34}$};
  \node[blue] (U44) at ( 5,0) {${45}$};
  \node[diamond, text=red, fill = red!20, minimum size= 2cm] (U23) at (2,1) {${24}$};
  \node[blue] (U34) at (4,1) {${35}$};
  \node[diamond, text=red, fill = red!20, minimum size= 2cm] (U24) at (3,2) {${25}$};
}
  \draw[thick] (M11) -- (M1) -- (M22) -- (M2) -- (M33) -- (M3) -- (M44) -- (M4) -- (M55);
  \draw[thick] (M23) -- (M33) -- (M34) -- (M44) -- (M45);
  \draw[thick] (M24) -- (M34) -- (M35);
  {\small
  \draw[thick] (M0) --node[blue, midway, above left]{$1$} (M11) --node[black, midway, above left]{$\mathbf{2}$} (M12) --node[red, midway, above right]{$\Sigma 1$} (M22) --node[blue, midway, above left]{$3$} (M23) --node[blue, midway, above left]{$4$} (M24) -- node[blue, midway, above left]{$5$} (M25) --node[red, midway, above right]{$\Sigma 2$} (M35) --node[blue, midway, above right]{$\Sigma 3$} (M45) --node[blue, midway, above right]{$\Sigma 4$} (M55) --node[blue, midway, above right]{$\Sigma 5$} (M5);
  }
  \draw[thick, red] (M12) -- (M22) -- (M2) -- (M33) -- (M34) -- (M35);
\end{tikzpicture}
\caption{For $\A=\CC A_5/\langle \alpha_1\alpha_2\rangle$ and $X = 2 \in \I_\A$, the compatible $Y \in \I_\A$ are shown in blue, and the incompatible are shown in red.}
\label{fig:urel2}
\end{figure}

\begin{rmk}
Given $X,Y \in \I_\A$ we can consider their \emph{compatibility degree} as $c(X,Y) = 0$ if $X$ and $Y$ are compatible (as in Definition~\ref{dfn: comp degree}) and $c(X,Y) = 1$ otherwise. In terms of representation theory, the compatibility degree $c(X,Y)$ is given by the dimensions of $\Hom$-groups in the derived category $K^b(\text{proj} \A)$. As defined in \cite{had}, it is given by $c(X,Y)=\dim(\Hom(X,\Sigma Y))+\dim(\Hom(Y,\Sigma X))$, where $\Sigma$ corresponds to shifting the complex in degree by $-1$.
\end{rmk}

Our configuration spaces generalize two interesting examples of \emph{binary geometries}. Namely, the partial compatifications of $\mathcal{M}_{0,n}$ in dihedral coordinates (see \cite{brown,lam}), and pellspaces (see \cite{pellytope}). To obtain $\widetilde{\mathcal{M}}_{0,n}$ one needs to consider $\A = \CC A_n$ the full path algebra, and for pellspaces $\A = \CC A_n/I$ where $I$ is the ideal generated by all paths of length two. 

\begin{rmk}[Dictionary with chords of polygon] 
We can translate the labeling of our grid to the language of chords in an $n+3$-gon. This might be more familiar in the context of $\mathcal{M}_{0,n}$ and the cluster realization of the associahedron. The indexing set $\I_{\CC A_n}$ is given by the labels $(i), (\Sigma i)$ for $i= 1,\ldots, n$ and $(i,j)$ for $1\leq i < j \leq n$. Each of them corresponds to a chord in an $n+3$-gon with ordered vertices $1,\ldots,n+3$ in the following way
\[
    (i) \leftrightarrow 1,i+1 \qquad (\Sigma i) \leftrightarrow i+1, n+3 \qquad (i,j) \leftrightarrow i+1,j+2.
\]
For example, for $n=3$ the grid labelled by chords in a hexagon is shown in Figure~\ref{fig:chords} (A). Translating the rules for compatibility in Definition~\ref{dfn: comp degree} to chords yields exactly the non-crossing chords conditions. Figure~\ref{fig:chords} shows this for the chord $24$. 
\begin{figure}[h!]
\begin{subfigure}{0.6\textwidth}
\centering
\begin{tikzpicture}[scale=1,every node/.style={inner sep=2pt}]
  \coordinate (M0) at (-3,-1);
  \coordinate (M1) at (-1,-1);
  \coordinate (M2) at (1,-1);
  \coordinate (M3) at (3,-1);
  \coordinate (M11) at (-2,0);
  \coordinate (M22) at ( 0,0);
  \coordinate (M33) at ( 2,0);
  \coordinate (M12) at (-1,1);
  \coordinate (M23) at ( 1,1);
  \coordinate (M13) at (0,2);
{\small
  \node[black] (U11) at (-1,0) {$\mathbf{ 24}$};
  \node[diamond, text=red, fill = red!20, minimum size= 2cm] (U23) at (1,0) {${35}$};
  \node[black] (U12) at (0,1) {${25}$};
\draw (M0) -- node[red, midway,above left]{$13$} (M11) -- (M1) -- (M22) --(M2) -- (M33) -- node[midway,above right]{$46$}  (M3);
\draw (M11) -- node[midway,above left]{$14$} (M12) -- (M22) -- (M23) -- node[red,midway,above right]{$36$} (M33);
\draw (M12) -- node[midway,above left]{$15$}(M13) -- node[midway,above right]{$26$} (M23);
\draw[thick, red] (M0) -- (M11);
\draw[thick, red] (M22) -- (M2) -- (M33) -- (M23) -- (M22);
}

\end{tikzpicture}
\caption{}
\end{subfigure}
\begin{subfigure}{0.3\textwidth}
\centering
\begin{tikzpicture}[scale=1.6, every node/.style={font=\small}]
  \foreach \i in {0,...,5}{
    \coordinate (P\i) at ({0+60*\i}:1);
  }
  \node[above right] at (P1) {$1$};
  \node[above left] at (P2) {$2$};
  \node[left] at (P3) {$3$};
  \node[below left] at (P4) {$4$};
  \node[below right] at (P5) {$5$};
  \node[right] at (P0) {$6$};
  
  \draw (P0) -- (P1) --(P2) --(P3) --(P4) --(P5) --(P0);
  \draw[thick] (P2) -- (P4);
  \draw[thick, red] (P3) -- (P1);
  \draw[thick, red] (P3) -- (P0);
  \draw[thick, red] (P3) -- (P5);
\end{tikzpicture}
\caption{}
\end{subfigure}
\caption{The grid picture for $\CC A_2$ (A) can be labelled by the chords of the regular hexagon (B). With this labelling two chords are compatible if and only if they do not cross.}
\label{fig:chords}
\end{figure}
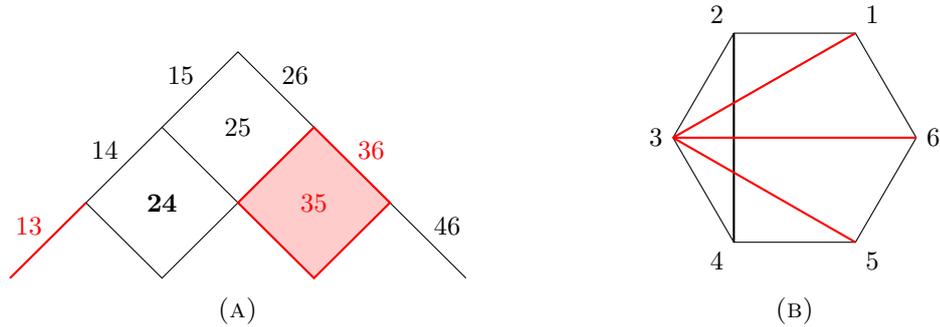

\noindent
For $\CC A_n$, it is known that the configuration spaces we consider, i.e. $\widetilde{\mathcal{M}}_{0,n+3}$, are closely related to the Deligne-Mumford-Knudsen compactification of $\mathcal{M}_{0,n+3}$. As explained in \cite[Section 1]{lam}, considering all possible orderings of the vertices of the $n+3$-gon gives an affine covering of $\overline{\mathcal{M}}_{0,n+3}$. Finding moduli interpretations for the varieties we obtain for quotient algebras is an interesting problem. For the pellspace corresponding to $\A = \CC A_3/\langle \alpha_2\alpha_2 \rangle$ this was initiated in \cite[Section 4.3]{pellytope} by interpreting that the points in $\PP^1$ were colliding at different speeds. We believe this can be generalized by working with moduli spaces of curves in weighted projective space as studied by Hassett \cite{hassett}.
\end{rmk}

\subsection{Monomial maps}
The poset of two-sided ideals of $\CC A_n$ corresponds functorialy to monomial maps between the varieties $U_{\A}$. Let us fix two Dyck paths, $\widetilde{D}$ and $D$, with $\widetilde{D} < D$. Their corresponding ideals, $I = I_{\widetilde D}$ and $J = I_{D}$ are ordered as $I \supset J$. Write $\A = \CC A_n / I$ and $\B = \CC A_n / J$ for the associated algebras. Then, we define a monomial map of coordinate rings
\[
\phi_{I \supset J} : \CC[U_\A] \longrightarrow \CC[U_\B].
\]
For the remaining of this section we fix $\A$ and $\mathcal{B}$ and denote $\phi_{I \supset J}$ as $\phi$ for simplicity. We write the coordinates in the ambient polynomial rings of $ \CC[U_\A]$ and $\CC[U_\B]$ as
\[
\CC[\widetilde{u}_{\widetilde X} \, | \, \widetilde{X} \in \I_\A],\qquad \CC[u_X \,|\, X \in \I_\B].
\]


\begin{dfn}\label{dfn: monomial maps}
    Let $\phi:  \CC[U_\A] \longrightarrow \CC[U_\B]$ be the map given in generators by the following rules. If $(i,j)$ is a diamond below $\widetilde{D}$ then $\phi(\widetilde{u}_{i,j}) = u_{i,j}$. For any up step $\widetilde{(i)}$ of $\widetilde{D}$, we consider the diamonds that are above $\widetilde{(i)}$ but below step $(i)$ of $D$, and map
        \[
        \phi: \widetilde{u}_{i} \mapsto u_i \prod_{\widetilde{i} < (k,l) < i} u_{k,l}.
        \]
    Similarly, for each down step $\widetilde{(\Sigma i)}$ of $\widetilde{D}$, we consider the diamonds above $\widetilde{(\Sigma i)}$ that are below down step $(\Sigma i)$ of $D$:
        \[
        \phi: \widetilde{u}_{\Sigma i} \mapsto u_{\Sigma i} \prod_{\widetilde{\Sigma i} < (k,l) < \Sigma i} u_{k,l}.
        \]
\end{dfn}

\begin{thm}\label{thm:surj}
The ring maps $\phi$ defined above are integral, i.e. $\phi^*$ is a surjective map of affine varieties. Moreover, for any ideals $I \supset J \supset K$ of $\CC A_n$, the maps satisfy
\[
\phi_{K \subset I} = \phi_{K \subset J} \circ \phi_{J \subset I}.
\]
\end{thm}
\begin{proof}
The composition property of the maps is immediate from the definition. The integrality of the monomial maps can therefore be checked step by step. It suffices to consider $\A$ and $\mathcal{B}$ such that $|\I_\A| = |\I_{\mathcal{B}}|-1$. Say the diamond that differs is $(i,j)$. Then, for all variables except $u_{i,j}, u_{i}$ and $u_{\Sigma j+1}$ it is trivial that they are integral. For the remaining three we use $u$-equations.
\end{proof}

\begin{rmk}
In terms of the algebras themselves, $\A$ is a quotient algebra of $\B$. Given a $\B$-module $M$, write $\pi M$ for the $\A$-module $M \otimes_{\B} \A$. This is a functor of module categories, $\pi: \text{mod}\B \rightarrow \text{mod}\A$, and it can be lifted to the categories of 2-term projective complexes, $\pi: K_\B \rightarrow K_\A$. Then the monomial map we defined above can be written as
\[
\phi(\widetilde{u}_X) = \prod_{Y \in \I_\B} \left(u_Y\right)^{[\pi Y: X]},
\]
where $[\pi Y:X]$ is the multiplicity of $X$ as an indecomposable summand of $\pi Y$. In particular, a diamond $(i,j)$ beneath the Dyck path $D$ corresponds to the object $X = (P_j \rightarrow P_i)$ of $K_\B$, where $P_i = e_i \B$ are the projective $\B$-modules. This object is sent to $\pi X = (\widetilde{P}_j \rightarrow \widetilde{P}_i)$, where $\widetilde{P}_i = e_i \A$ are the projective $\A$-modules. But $\pi X$ is decomposable in $K_\A$ if and only if $\alpha_{j-1} \alpha_{j-2} \cdots \alpha_{i} \in I$. Equivalently, $\pi X$ is decomposable if the diamond $(i,j)$ lies above the Dyck path $\widetilde{D}$. In this case
\[
\pi X \simeq (\widetilde{P}_j \rightarrow 0) \oplus (0 \rightarrow \widetilde{P}_i)
\]
in $K_\A$, and so $u_X$ appears in the monomial formula for both $\widetilde{u}_{\Sigma \widetilde{P}_j}$ and $\widetilde{u}_{\widetilde{P}_i}$. These monomial maps are defined and studied for all finite representation type algebras in \cite{had}.
\end{rmk}

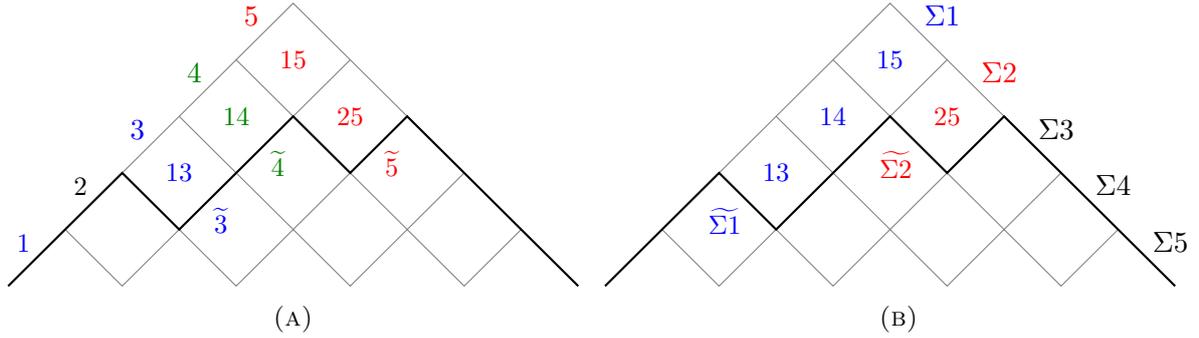
\begin{figure}
\begin{subfigure}[t]{0.49\textwidth}
\centering
\begin{tikzpicture}[scale=0.75,every node/.style={inner sep=2pt}]
  \coordinate (M0) at (-3,-1);
  \coordinate (M1) at (-1,-1);
  \coordinate (M2) at (1,-1);
  \coordinate (M3) at (3,-1);
  \coordinate (M4) at (5,-1);
  \coordinate (M5) at (7,-1);
  \coordinate (M11) at (-2,0);
  \coordinate (M22) at ( 0,0);
  \coordinate (M33) at ( 2,0);
  \coordinate (M44) at ( 4,0);
  \coordinate (M55) at ( 6,0);
  \coordinate (M12) at (-1,1);
  \coordinate (M23) at ( 1,1);
  \coordinate (M34) at ( 3,1);
  \coordinate (M45) at ( 5,1);
  \coordinate (M13) at (0,2);
  \coordinate (M24) at (2,2);
  \coordinate (M35) at (4,2);
  \coordinate (M14) at (1,3);
  \coordinate (M25) at (3,3);
  \coordinate (M15) at (2,4);
{\small
  \node[blue] (U12) at (0,1) {${13}$};
  \node[green!50!black] (U13) at (1,2) {${14}$};
  \node[red] (U24) at (3,2) {${25}$};
  \node[red] (U14) at (2,3) {${15}$};
}
  \draw[gray] (M0) -- (M11) -- (M1) -- (M22) -- (M2) -- (M33) -- (M3) -- (M44) -- (M4) -- (M55) -- (M5);
  \draw[gray] (M11) -- (M12) -- (M22) -- (M23) -- (M33) -- (M34) -- (M44) -- (M45) -- (M55);
  \draw[gray] (M12) --node[blue,midway,above left]{$3$} (M13) -- (M23) -- (M24) -- (M34) -- (M35) -- (M45);
  \draw[gray] (M13) --node[green!50!black,midway,above left]{$4$} (M14) -- (M24) -- (M25) -- (M35);
  \draw[gray] (M14) --node[red,midway,above left]{$5$} (M15) -- (M25);
  {\small
  \draw[thick] (M0)
  --node[blue, midway, above left]{$1$} (M11) 
  --node[black, midway, above left]{${2}$} (M12) 
  -- (M22) 
  --node[blue, midway, below right]{$\widetilde{3}$} (M23) 
  --node[green!50!black, midway, below right]{$\widetilde{4}$} (M24)
  -- (M34) 
  --node[red, midway, below right]{$\widetilde{5}$} (M35) 
  -- (M45) 
  -- (M55) 
  -- (M5);
  }
\end{tikzpicture}
\caption{}
\label{fig:monmap1}
\end{subfigure}
\begin{subfigure}[t]{0.49\textwidth}
\centering
\begin{tikzpicture}[scale=0.75,every node/.style={inner sep=2pt}]
  \coordinate (M0) at (-3,-1);
  \coordinate (M1) at (-1,-1);
  \coordinate (M2) at (1,-1);
  \coordinate (M3) at (3,-1);
  \coordinate (M4) at (5,-1);
  \coordinate (M5) at (7,-1);
  \coordinate (M11) at (-2,0);
  \coordinate (M22) at ( 0,0);
  \coordinate (M33) at ( 2,0);
  \coordinate (M44) at ( 4,0);
  \coordinate (M55) at ( 6,0);
  \coordinate (M12) at (-1,1);
  \coordinate (M23) at ( 1,1);
  \coordinate (M34) at ( 3,1);
  \coordinate (M45) at ( 5,1);
  \coordinate (M13) at (0,2);
  \coordinate (M24) at (2,2);
  \coordinate (M35) at (4,2);
  \coordinate (M14) at (1,3);
  \coordinate (M25) at (3,3);
  \coordinate (M15) at (2,4);
{\small
  \node[blue] (U12) at (0,1) {${13}$};
  \node[blue] (U13) at (1,2) {${14}$};
  \node[red] (U24) at (3,2) {${25}$};
  \node[blue] (U14) at (2,3) {${15}$};
}
  \draw[gray] (M0) -- (M11) -- (M1) -- (M22) -- (M2) -- (M33) -- (M3) -- (M44) -- (M4) -- (M55) --node[black,midway,above right]{$\Sigma 5$} (M5);
  \draw[gray] (M11) -- (M12) -- (M22) -- (M23) -- (M33) -- (M34) -- (M44) -- (M45) --node[black,midway,above right]{$\Sigma 4$} (M55);
  \draw[gray] (M12) -- (M13) -- (M23) -- (M24) -- (M34) -- (M35) --node[black,midway,above right]{$\Sigma 3$} (M45);
  \draw[gray] (M13) -- (M14) -- (M24) -- (M25) -- node[red,midway,above right]{$\Sigma 2$} (M35);
  \draw[gray] (M14) -- (M15) -- node[blue,midway,above right]{$\Sigma 1$}(M25);
  {\small
  \draw[thick] (M0)
  -- (M11) 
  --(M12) 
  --node[blue, midway, below left]{$\widetilde{\Sigma 1}$} (M22) 
  --(M23) 
  --(M24)
  --node[red, midway, below left]{$\widetilde{\Sigma 2}$} (M34) 
  -- (M35) 
  -- (M45) 
  -- (M55) 
  -- (M5);
  }
\end{tikzpicture}
\caption{}
\label{fig:monmap2}
\end{subfigure}
\caption{The monomial map $\phi_{I \supset J}$ for $J = \emptyset$ and $I = \langle \alpha_2\alpha_1, \alpha_4\alpha_3\alpha_2\rangle$ is given by monomials for each of the up steps, (A), and monomials for each of the down steps, (B), of the new Dyck path.} 
\label{fig:monmap}
\end{figure}

\begin{ex}
For $n=5$, consider the ideals $I \supset J$ with $I = \langle \alpha_2\alpha_1, \alpha_4\alpha_3\alpha_2\rangle$ and $J = \emptyset$. The Dyck path $D$ for $I$ is shown in bold in Figure \ref{fig:monmap}. As above, take $\A = \CC A_5 / I$ and $\B = \CC A_5$ to be the associated algebras. Then we have a monomial map $\phi_{I\supset J} : \CC[U_\A] \rightarrow \CC[U_\B]$. On the variables for the up steps of $D$, the map is (Figure \ref{fig:monmap1})
\[
\widetilde{u}_{\widetilde{1}} \mapsto u_1, \qquad \widetilde{u}_{\widetilde{2}} \mapsto u_2, \qquad \widetilde{u}_{\widetilde{3}} \mapsto u_3 u_{13},\qquad \widetilde{u}_{\widetilde 4} \mapsto u_4 u_{14},\qquad \widetilde{u}_{\widetilde{5}} \mapsto u_5 u_{15} u_{25}.
\]
On the variables for the down steps of $D$, the map is (Figure \ref{fig:monmap2})
\[
\widetilde{u}_{\widetilde{\Sigma 1}} \mapsto u_{\Sigma 1} u_{13} u_{14} u_{15},\qquad \widetilde{u}_{\widetilde{\Sigma 2}} \mapsto u_{\Sigma 2} u_{25},\qquad \widetilde{u}_{\widetilde{\Sigma 3}} \mapsto u_{\Sigma 3},\qquad \widetilde{u}_{\widetilde{\Sigma 4}} \mapsto u_{\Sigma 4},\qquad \widetilde{u}_{\widetilde{\Sigma 5}} \mapsto u_{\Sigma 5}.
\]
\end{ex}

\subsection{Positive parts and boundary divisors}
For each algebraic variety $U_{\A}$ we can consider a natural positive (resp. nonnegative) part $(U_{\A})_{>0}$ (resp. $(U_{\A})_{\geq0}$) by  intersecting it with the positive orthant $(\RR_+)^{\I_{\A}}$ (resp. $(\RR_{\geq 0})^{\I_{\A}}$). In this subsection we describe the prime divisors of our variety which bound $(U_{\A})_{\geq0}$ and describe how they behave under the monomial maps.  

\begin{prop}\label{prop:fact}
Let $\A = \CC A_n/I_D$ and $X \in \I_\A$. Denote $U_\A(X)$ for the subvariety of $U_\A$ given by $u_X =0$. Then, $U_\A(X)$ is isomorphic to $U_{\A'}\times U_{\A''}$ where $\A'$ and $\A''$ are linear Nakayama algebras of lower rank (i.e. $\CC A_m/J$ for $m<n$). Concretely, we have the following cases:
\smallskip
\begin{itemize}
    \item if $X = (i,j)$ then $\A' = \CC A_{j-i-1}$ and $\A'' = \CC A_{n+i-j}/I_{D_{ij}}$
    \smallskip
    \item if $X = (i)$ or $X = (\Sigma i)$ then $\A' = \CC A_{n-i}/I_{D_R}$ and $\A'' = \CC A_{i-1}/I_{D_{L}}$
\end{itemize}
\smallskip
where $D_{ij},D_R$ and $D_L$ are Dyck paths in the corresponding grids which we obtain from $D$.
\end{prop}
\begin{proof}
    The proposition follows from Definition~\ref{dfn: comp degree} understood in our (truncated) grid. Note that the $u$-equations have the property that once we set $u_X=0$ then for all $Y \in \I_\A$ incompatible with $X$ we have $u_Y = 1$. Then, it remains to understand what the relations between the remaining $Z \in \I_\A$ are. By cutting out the diamonds of the grid which are incompatible with $X$ we are left with two truncated grids. In Figures~\ref{fig:ijbound} and \ref{fig:ibound} the two cases are shown.

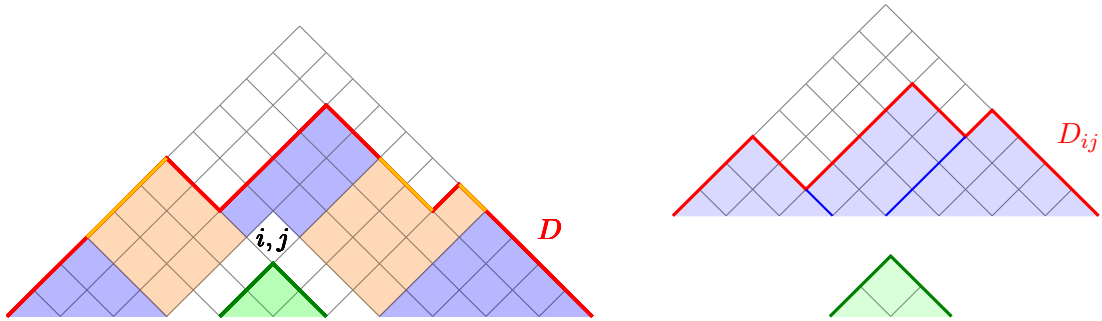
\begin{figure}[h!]
\begin{subfigure}{0.48\textwidth}
\begin{tikzpicture}[scale=0.35]
\def\N{11}      
\def\x0{-3}    
\def\y0{-1}    
\def\xstep{2}  
\def\ystep{1}  
\foreach \r in {0,...,\N}{
  \pgfmathtruncatemacro{\len}{\N-\r}
  \foreach \k in {0,...,\len}{
    \coordinate (M-\r-\k) at ({\x0 + \xstep*\k + \r}, {\y0 + \ystep*\r});
  }
}
\foreach \r in {0,...,\numexpr\N-1\relax}{
  \pgfmathtruncatemacro{\len}{\N-\r}
  \ifnum\len>0
    \foreach \k in {0,...,\numexpr\len-1\relax}{
      \draw[gray] (M-\r-\k) -- (M-\the\numexpr\r+1\relax-\k);
      \draw[gray] (M-\r-\the\numexpr\k+1\relax) -- (M-\the\numexpr\r+1\relax-\k);
    }
  \fi
\node at (M-4-3) [below=3pt] {\small $i,j$};
\filldraw[orange,opacity=0.03] (M-3-4) -- (M-0-7) -- (M-4-7) -- (M-5-6) -- (M-4-6) -- (M-6-4) -- (M-3-4);
\filldraw[orange,opacity=0.03] (M-3-3) -- (M-0-3) -- (M-3-0) -- (M-6-0) -- (M-3-3);
\filldraw[blue,opacity=0.03] (M-0-0) -- (M-3-0) -- (M-0-3) -- (M-0-0);
\filldraw[blue,opacity=0.03] (M-3-3) -- (M-4-2) -- (M-8-2) -- (M-6-4) -- (M-3-4) -- (M-4-3) -- (M-3-3);
\filldraw[blue,opacity=0.03] (M-0-7) -- (M-4-7) -- (M-0-11) -- (M-0-7);
\filldraw[green, opacity=0.03] (M-0-4) -- (M-2-4) -- (M-0-6) -- (M-0-4);
\draw[very thick,green!50!black] (M-0-4) -- (M-2-4) -- (M-0-6);
\draw[very thick,red] (M-0-0) -- (M-6-0) -- (M-4-2) -- (M-8-2) -- (M-4-6) -- (M-5-6) --node[midway,above right]{$D$} (M-0-11);
\draw[very thick,orange!50!yellow] (M-3-0) -- (M-6-0);
\draw[very thick,orange!50!yellow] (M-6-4) -- (M-4-6);
\draw[very thick,orange!50!yellow] (M-5-6) -- (M-4-7);
}
\end{tikzpicture}
\end{subfigure}
\begin{subfigure}{0.48\textwidth}
\centering
\begin{tikzpicture}[scale=0.35]
\def\N{8}      
\def\x0{-3}    
\def\y0{-1}    
\def\xstep{2}  
\def\ystep{1}  
\foreach \r in {0,...,\N}{
  \pgfmathtruncatemacro{\len}{\N-\r}
  \foreach \k in {0,...,\len}{
    \coordinate (M-\r-\k) at ({\x0 + \xstep*\k + \r}, {\y0 + \ystep*\r});
  }
}
\foreach \r in {0,...,\numexpr\N-1\relax}{
  \pgfmathtruncatemacro{\len}{\N-\r}
  \ifnum\len>0
    \foreach \k in {0,...,\numexpr\len-1\relax}{
      \draw[gray] (M-\r-\k) -- (M-\the\numexpr\r+1\relax-\k);
      \draw[gray] (M-\r-\the\numexpr\k+1\relax) -- (M-\the\numexpr\r+1\relax-\k);
    }
  \fi
}
\filldraw[blue,opacity=0.15] (M-0-0) -- (M-3-0) -- (M-1-2) -- (M-5-2) -- (M-3-4) -- (M-4-4) -- (M-0-8) -- (M-0-0);
\draw[very thick,red] (M-0-0) -- (M-3-0) -- (M-1-2) -- (M-5-2) -- (M-3-4) -- (M-4-4) --node[midway,above right]{$D_{ij}$} (M-0-8);
\draw[thick,blue] (M-1-2) -- (M-0-3);
\draw[thick,blue] (M-0-4) -- (M-3-4);
\end{tikzpicture}
\\~\\
\begin{tikzpicture}[scale=0.4]
\def\N{2}      
\def\x0{-3}    
\def\y0{-1}    
\def\xstep{2}  
\def\ystep{1}  
\foreach \r in {0,...,\N}{
  \pgfmathtruncatemacro{\len}{\N-\r}
  \foreach \k in {0,...,\len}{
    \coordinate (M-\r-\k) at ({\x0 + \xstep*\k + \r}, {\y0 + \ystep*\r});
  }
}
\foreach \r in {0,...,\numexpr\N-1\relax}{
  \pgfmathtruncatemacro{\len}{\N-\r}
  \ifnum\len>0
    \foreach \k in {0,...,\numexpr\len-1\relax}{
      \draw[gray] (M-\r-\k) -- (M-\the\numexpr\r+1\relax-\k);
      \draw[gray] (M-\r-\the\numexpr\k+1\relax) -- (M-\the\numexpr\r+1\relax-\k);
    }
  \fi
}
\filldraw[green,opacity=0.15] (M-0-0) -- (M-2-0) -- (M-0-2) -- (M-0-0);
\draw[very thick,green!50!black] (M-0-0) -- (M-2-0) -- (M-0-2);
\end{tikzpicture}
\end{subfigure}
\caption{For a Dyck path $D$ and diamond $(i,j)$, the incompatible diamonds and edges are shown in orange (left). On the divisor $u_{i,j} = 0$, the variety $U_\A$ associated to $D$ is isomorphic to a product of two such varieties: one for the Dyck path $D_{i,j}$ (right top) obtained by contracting the orange region, and one for the triangle below $(i,j)$ (right bottom). The white diamonds (left) label the steps of the green Dyck path (right bottom).}
\label{fig:ijbound}
\end{figure}

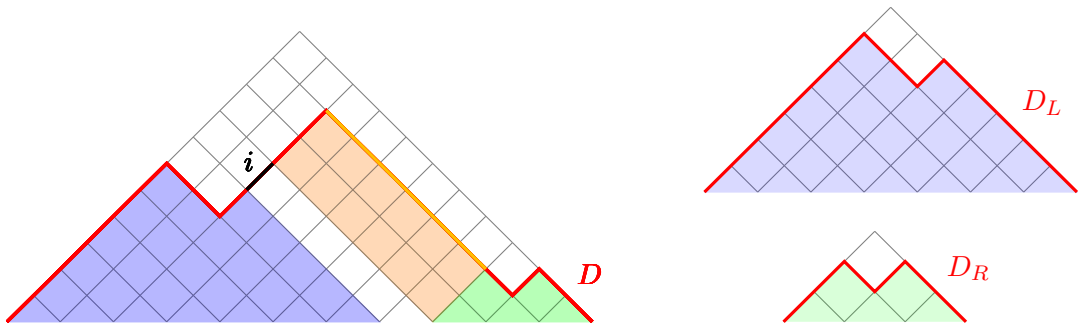
\begin{figure}[h!]
\begin{subfigure}{0.48\textwidth}
\begin{tikzpicture}[scale=0.35]
\def\N{11}      
\def\x0{-3}    
\def\y0{-1}    
\def\xstep{2}  
\def\ystep{1}  
\foreach \r in {0,...,\N}{
  \pgfmathtruncatemacro{\len}{\N-\r}
  \foreach \k in {0,...,\len}{
    \coordinate (M-\r-\k) at ({\x0 + \xstep*\k + \r}, {\y0 + \ystep*\r});
  }
}
\foreach \r in {0,...,\numexpr\N-1\relax}{
  \pgfmathtruncatemacro{\len}{\N-\r}
  \ifnum\len>0
    \foreach \k in {0,...,\numexpr\len-1\relax}{
      \draw[gray] (M-\r-\k) -- (M-\the\numexpr\r+1\relax-\k);
      \draw[gray] (M-\r-\the\numexpr\k+1\relax) -- (M-\the\numexpr\r+1\relax-\k);
    }
  \fi
\filldraw[orange,opacity=0.03] (M-6-2) -- (M-0-8) -- (M-2-8) -- (M-8-2) -- (M-6-2);
\filldraw[blue,opacity=0.03] (M-0-0) -- (M-6-0) -- (M-4-2) -- (M-5-2) -- (M-0-7) -- (M-0-0);
\filldraw[green,opacity=0.03] (M-0-8) -- (M-2-8) -- (M-1-9) -- (M-2-9) -- (M-0-11) -- (M-0-8);
\draw[very thick,red] (M-0-0) -- (M-6-0) -- (M-4-2) -- (M-5-2) --node[midway,above left=-2pt]{\color{black} $i$} (M-6-2) -- (M-8-2) -- (M-1-9) -- (M-2-9) --node[midway,above right]{$D$} (M-0-11);
\draw[very thick,black] (M-5-2) -- (M-6-2);
\draw[very thick,orange!50!yellow] (M-8-2) -- (M-2-8);
}
\end{tikzpicture}
\end{subfigure}
\begin{subfigure}{0.48\textwidth}
\centering
\begin{tikzpicture}[scale=0.35]
\def\N{7}      
\def\x0{-3}    
\def\y0{-1}    
\def\xstep{2}  
\def\ystep{1}  
\foreach \r in {0,...,\N}{
  \pgfmathtruncatemacro{\len}{\N-\r}
  \foreach \k in {0,...,\len}{
    \coordinate (M-\r-\k) at ({\x0 + \xstep*\k + \r}, {\y0 + \ystep*\r});
  }
}
\foreach \r in {0,...,\numexpr\N-1\relax}{
  \pgfmathtruncatemacro{\len}{\N-\r}
  \ifnum\len>0
    \foreach \k in {0,...,\numexpr\len-1\relax}{
      \draw[gray] (M-\r-\k) -- (M-\the\numexpr\r+1\relax-\k);
      \draw[gray] (M-\r-\the\numexpr\k+1\relax) -- (M-\the\numexpr\r+1\relax-\k);
    }
  \fi
}
\filldraw[blue,opacity=0.15] (M-0-0) -- (M-6-0) -- (M-4-2) -- (M-5-2) -- (M-0-7) -- (M-0-0);
\draw[very thick,red] (M-0-0) -- (M-6-0) -- (M-4-2) -- (M-5-2) --node[midway, above right]{$D_L$} (M-0-7);
\end{tikzpicture}
\\~\\
\begin{tikzpicture}[scale=0.4]
\def\N{3}      
\def\x0{-3}    
\def\y0{-1}    
\def\xstep{2}  
\def\ystep{1}  
\foreach \r in {0,...,\N}{
  \pgfmathtruncatemacro{\len}{\N-\r}
  \foreach \k in {0,...,\len}{
    \coordinate (M-\r-\k) at ({\x0 + \xstep*\k + \r}, {\y0 + \ystep*\r});
  }
}
\foreach \r in {0,...,\numexpr\N-1\relax}{
  \pgfmathtruncatemacro{\len}{\N-\r}
  \ifnum\len>0
    \foreach \k in {0,...,\numexpr\len-1\relax}{
      \draw[gray] (M-\r-\k) -- (M-\the\numexpr\r+1\relax-\k);
      \draw[gray] (M-\r-\the\numexpr\k+1\relax) -- (M-\the\numexpr\r+1\relax-\k);
    }
  \fi
}
\filldraw[green,opacity=0.15] (M-0-0) -- (M-2-0) -- (M-1-1) -- (M-2-1) -- (M-0-3);
\draw[very thick,red] (M-0-0) -- (M-2-0) -- (M-1-1) -- (M-2-1) --node[midway, above right]{$D_R$} (M-0-3);
\end{tikzpicture}
\end{subfigure}
\caption{For a Dyck path $D$ and up step $(i)$, the incompatible diamonds and edges are shown in orange (left). On the divisor $u_{i} = 0$, the variety $U_\A$ associated to $D$ is isomorphic to a product of two such varieties, for the Dyck paths $D_{L}$ and $D_R$ (right) obtained by contracting the orange region.}
\label{fig:ibound}
\end{figure}

\end{proof}

\begin{ex}
    Continuing Example~\ref{ex: A5 urels} we consider $\A = \CC A_5 / \langle \alpha_2\alpha_1 \rangle$. The boundary divisors of $U_{\A}$ associated to setting $u_{24}=0$ and $u_2=0$ are given by 
    \[
        U_{\A}(u_{24}) \cong U_{\CC A_1} \times U_{\CC A_3/D_{24}} \qquad U_{\A}(u_{2}) \cong U_{\CC A_1} \times U_{\CC A_2}
    \]
    where $D_{24}$ is the \textcolor{red}{red} Dyck path  \begin{tikzpicture}[baseline=-0.6ex,scale=0.3]
  \coordinate (M11) at (-2,0);
  \coordinate (M22) at ( 0,0);
  \coordinate (M33) at ( 2,0);
  \coordinate (M12) at (-1,1);
  \coordinate (M23) at ( 1,1);
  \coordinate (M13) at (0,2);
  \coordinate (P1) at (-3,-1);
  \coordinate (P2) at (-1,-1);
  \coordinate (P3) at (1,-1) ;
  \coordinate (P4) at (3,-1) ;
    \draw[gray] (M12) -- (M22) -- (P2) -- (M11) --  (M12);
  \draw[gray] (M22) -- (M23) -- (M33) -- (P3) --  (M22);
  \draw[gray] (M12) -- (M13) -- (M23) -- (M22) -- (M12);
    \draw[thick, red] (P1) -- (M11) -- (P2) -- (M22) -- (M23) -- (M33) -- (P4);
  \end{tikzpicture} in the $A_3$ grid. These divisors are obtained by removing the red boxes in the truncated grids drawn in Figure~\ref{fig:urel24} and Figure~\ref{fig:urel2} and piecing together the remaining diamonds. 
\end{ex}

\begin{rmk}
    Note that the monomial maps from Definition~\ref{dfn: monomial maps} behave well when restricted to the boundary divisors. Consider the geometric map $\phi^*:  U_\B \longrightarrow U_\A$ and denote the ambient coordinates of $U_\A$ and  $U_\B$ as $\widetilde{u}_{\widetilde{X}}$ and $u_X$ respectively. Then we have
    \[
        U_\A(\widetilde{u}_{\widetilde{X}}) = \begin{dcases}
            \phi^*(U_\B(u_X)) &\text{if } \widetilde{X} = (i,j) \text{ is below } \widetilde{D},\\
            \bigcup_{\widetilde{i} < (k,l) < i} \phi^*(U_\B(u_{k,l})) &\text{if } \widetilde{X}=\widetilde{(i)},\\
            \bigcup_{\widetilde{\Sigma i} < (k,l) < \Sigma i} \phi^*(U_\B(u_{k,l})) &\text{if } \widetilde{X}=\widetilde{(\Sigma i)}.
        \end{dcases}
    \]
\end{rmk}


\section{F-polynomials and toric coordinates}\label{sec:F}

In this section we give an explicit parametrisation of the varieties $U_\A^\circ$ and $U_\A$ in terms of the $F$-polynomials of the indecomposable modules of $\A$, see \cite{derksen2010quivers} for the general definition. This allows us to see $U_\A$ as an affine open subset of a projective toric variety. 

\smallskip

We consider variables $y_1,\ldots,y_n$ and denote the polynomial ring on these variables by $\CC[y]$ and the corresponding field of rational functions as $\CC(y)$. To simplify notation we write $y_{i,\ldots,j}$ instead of the monomial $y_iy_{i+1}\cdots y_j$ for $1\leq i < j \leq n$. Given a $\CC A_n$-module $M_{i,j}$, its associated $F$-polynomial is $F_{i,j} := 1+y_i + y_{i}y_{i+1} + \cdots + y_{i,\ldots,j} \in \CC[y].$

\smallskip

Fix a quotient algebra $\A$ with corresponding Dyck path $D$. Recall that the ambient space of $U_\A$ has coordinates indexed by the indexing set $\I_\A$. Let the up steps of $D$ have coordinates $(j_i,i-1) \to (j_i, i)$; and the down steps $(i,k_i) \to (i+1, k_i)$. For each $X \in \I_\A$ we consider a rational function $f_X \in \CC(y)$ given as follows
\begin{align}\label{eq: param}
    &f_{i,j} := \frac{F_{i,j}F_{i+1,j-1}}{F_{i+1,j} F_{i,j-1}}, && \text{for $(i,j)$ diamond below $D$} \\
    &f_{i} := \frac{F_{j_i,i-1}}{F_{j_i,i}}, \quad \text{and} \quad f_{\Sigma i} := y_i\frac{F_{i+1,k_i}}{F_{i,k_i}} && \text{for } 1\leq i \leq n.\nonumber
\end{align}
By convention if $k=0, l=n+1$ or $k>l$ we set $F_{k,l} = 1$. We denote $R_\A$ (resp. $R_\A^\circ$) for the subring of $\CC(y)$ generated by $f_X$ (resp. $f_X^{\pm 1}$) for $X \in \I_\A$. These rational functions can be obtained from our grid by labelling the lattice points with the $F$-polynomials and considering for each diamond, labelled by $X=(i,j)$, the product $f_X = \frac{\text{up}\cdot\text{down}}{\text{left}\cdot\text{right}}$. The up steps and down steps follow a similar rule but we add a $y_i$ factor for $(\Sigma i)$.

\begin{ex}
For $n=4$, consider the grids labelled by $F$-polynomials for $\CC A_4$ and for the quotient algebra $\A = \CC A_4/I$ with $I = \langle \alpha_2\alpha_2 \rangle$:
    \begin{figure}[H]
    \begin{subfigure}{0.48\textwidth}
        \centering
\begin{tikzpicture}[scale=0.8,every node/.style={inner sep=2pt}]
  \coordinate (M0) at (-3,-1);
  \coordinate (M1) at (-1,-1);
  \coordinate (M2) at (1,-1);
  \coordinate (M3) at (3,-1);
  \coordinate (M4) at (5,-1);
  \node (M11) at (-2,0) {$F_{11}$};
  \node (M22) at ( 0,0) {$F_{22}$};
  \node (M33) at ( 2,0) {$F_{33}$};
  \node (M44) at ( 4,0) {$F_{44}$};
  \node (M12) at (-1,1) {$F_{12}$};
  \node (M23) at ( 1,1) {$F_{23}$};
  \node (M34) at ( 3,1) {$F_{34}$};
  \node (M13) at (0,2) {$F_{13}$};
  \node (M24) at (2,2) {$F_{24}$};
  \node (M14) at (1,3) {$F_{14}$};
  \draw[thick] (M0) -- (M11) -- (M1) -- (M22) -- (M2) -- (M33) -- (M3) -- (M44) -- (M4);
  \draw[thick] (M11) -- (M12) -- (M22) -- (M23) -- (M33) -- (M34) -- (M44);
  \draw[thick] (M12) -- (M13) -- (M23) -- (M24) -- (M34);
  \draw[thick] (M13) -- (M14) -- (M24);
\end{tikzpicture}
\caption{Grid for $\CC A_4$}
    \end{subfigure}
    \begin{subfigure}{0.48\textwidth}
        \centering
        \begin{tikzpicture}[scale=0.8,every node/.style={inner sep=2pt}]
  \coordinate (M0) at (-3,-1);
  \coordinate (M1) at (-1,-1);
  \coordinate (M2) at (1,-1);
  \coordinate (M3) at (3,-1);
  \coordinate (M4) at (5,-1);
  \node (M11) at (-2,0) {$F_{11}$};
  \node (M22) at ( 0,0) {$F_{22}$};
  \node (M33) at ( 2,0) {$F_{33}$};
  \node (M44) at ( 4,0) {$F_{44}$};
  \node (M12) at (-1,1) {$F_{12}$};
  \node (M23) at ( 1,1) {$F_{23}$};
  \node (M34) at ( 3,1) {$F_{34}$};
  \node (M24) at (2,2) {$F_{24}$};
  \draw[thick] (M0) -- (M11) -- (M1) -- (M22) -- (M2) -- (M33) -- (M3) -- (M44) -- (M4);
  \draw[thick] (M11) -- (M12) -- (M22) -- (M23) -- (M33) -- (M34) -- (M44);
  \draw[very thick, red] (M0) -- (M11) -- (M12) -- (M22) -- (M23) -- (M24) -- (M34) -- (M44) -- (M4);
\end{tikzpicture}
        \caption{Grid for $\A = \CC A_4 / \langle \alpha_2\alpha_1 \rangle$}
    \end{subfigure}
    \end{figure}
    
\noindent
    $\A$ has $8$ indecomposable modules which label the diamonds and the up-steps of $D$ in our grid. We also consider the four extra labels of the down steps of $D$. Then, in this example the $12$ rational functions $f_X$ from \eqref{eq: param} are 
    \begin{align*}
        f_{1,2} &= \frac{F_{1,2}}{F_{2,2}F_{1,1}} & f_{2,3} &= \frac{F_{2,3}}{F_{3,3}F_{2,2}} & f_{\Sigma 1} & = y_1\frac{F_{2,2}}{F_{1,2}} & f_{1}&=y\frac{1}{F_{1,1}}\\
        f_{2,4} &=  \frac{F_{2,4}F_{3,3}}{F_{2,3}F_{3,4}}  & f_{3,4} &= \frac{F_{3,4}}{F_{3,3}} & f_{\Sigma 2} & = y_2\frac{F_{3,4}}{F_{2,4}} & f_{2}&=\frac{F_{1,1}}{F_{1,2}}\\
        f_{\Sigma 3} & = y_3\frac{F_{4,4}}{F_{3,4}} & f_{3}&=\frac{F_{2,2}}{F_{2,3}} & f_{\Sigma 4} & = y_4\frac{1}{F_{4,4}} & f_{4}&=\frac{F_{2,3}}{F_{2,4}}.
   \end{align*}
\end{ex}

\begin{thm}\label{thm:param}
    The maps $\varphi^\circ: \CC[U_\A^\circ] \to R_\A^\circ$ and $\varphi: \CC[U_\A] \to R_\A$ given by $u_X \to f_X$ are ring isomorphisms. In other words, they yield parameterizations of the very affine variety $U_\A^\circ$ and the affine variety $U_\A$. 
\end{thm}

\begin{proof}
    We first see that the rational functions $f_X$ satisfy the $u$-equations defining $U_\A$ (and $U_\A^\circ$). A direct calculation shows that the $F$-polynomials satisfy the indentity $F_{i,j-1} F_{i+1,j} = F_{i,j}F_{i+1,j-1} + y_{i+1,j}$. Plugging this into the definition of $f_X$ yields
\[
1-f_{i,j} = \frac{y_{i+1,j}}{F_{i,j-1}F_{i+1,j}},\qquad 1-f_{i} = \frac{y_{j_i,i}}{F_{i,j_i}},\qquad 1-f_{\Sigma i} = \frac{1}{F_{i,k_i}}.
\]
Now we want to see these expressions are exactly the monomials in the $f_X$ given by the compatibility condition of the $u$-equations. From our grid it is easier to see the telescopic cancellations in the following products:
\begin{equation}\label{eq: tellescopic canc}
    \frac{y_{i,j}}{F_{i,j}} = \prod_{\substack{X \in \text{ right rectangle of } (i,j)\\ X \neq (k)}} f_X \qquad \qquad \frac{1}{F_{i,j}} = \prod_{\substack{X \in \text{ left rectangle of } (i,j)\\ X \neq \Sigma (k)}} f_X.
\end{equation}
The key observation is that when we consider a rectangle $r$ as in Figure~\ref{fig: hills and dents} in our grid shaved-off on the top by a Dyck path, we have the following formula
\[
    \prod_{X \text{ label in } r} f_X = \frac{\text{peaks}}{\text{valleys}}
\]
where we abuse terminology and also refer to the left and right corners of $r$ as valleys.
\begin{figure}[h!]
    \centering
\begin{tikzpicture}[scale=0.35]
\def\N{11}      
\def\x0{-3}    
\def\y0{-1}    
\def\xstep{2}  
\def\ystep{1}  
\foreach \r in {0,...,\N}{
  \pgfmathtruncatemacro{\len}{\N-\r}
  \foreach \k in {0,...,\len}{
    \coordinate (M-\r-\k) at ({\x0 + \xstep*\k + \r}, {\y0 + \ystep*\r});
  }
}
\filldraw[green,opacity=0.15] (M-0-5) -- (M-4-1) -- (M-5-1) -- (M-4-2) --(M-5-2) -- (M-4-3) -- (M-6-3) -- (M-4-5) -- (M-0-5);
\draw[green!50!black] (M-0-5) -- (M-4-1) -- (M-5-1) -- (M-4-2) --(M-5-2) -- (M-4-3) -- (M-6-3) -- (M-4-5) -- (M-0-5);
\draw[thick,red] (M-0-0) -- (M-6-0) -- (M-4-2) -- (M-5-2) -- (M-4-3) -- (M-6-3) -- (M-3-6) -- (M-4-6) -- (M-2-8) -- (M-3-8) --node[midway, above right]{$D$} (M-0-11); 
{\color{blue}
\node at (M-5-1) {$\bullet$};
\node at (M-5-2) {$\bullet$};
\node at (M-6-3) {$\bullet$};
}
{\color{red}
\node at (M-4-1) {$\bullet$};
\node at (M-4-2) {$\bullet$};
\node at (M-4-3) {$\bullet$};
\node at (M-4-5) {$\bullet$};
}
\node[green!50!black] at (M-2-4) {$r$};
\end{tikzpicture}
    \caption{The product of $f_{ij}$ in the region $r$ (green) is equal to the product of $F_{ij}$ for the \emph{peaks} (blue) divided by the $F_{ij}$ for the \emph{valleys} (red).}
    \label{fig: hills and dents}
\end{figure}
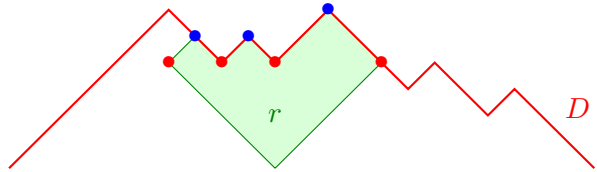
This formula follows from telescopic cancellations. For a $2\times 2$ rectangle as in Figure~\ref{fig: rectangle cancellation} (A), the product of $f_X$ in the rectangle is $f_{M_1}f_{M_2}f_{M_3}f_{M_4} = \frac{F_{k-1,l+1}F_{k+1,l-1}}{F_{k+1,l+1}F_{k-1,l-1}}$. We can iterate this and obtain the same $\frac{\text{up}\cdot\text{down}}{\text{left}\cdot\text{right}}$ formula for bigger rectangles. If we have some boxes shaved-off from the top of the rectangle by the Dyck path, then we must add the $F$'s labeling the peaks in the numerator and the $F$'s labeling the valleys in the denominator to obtain the correct formula. For example, in Figure~\ref{fig: rectangle cancellation} (B) we have $f_{M_1}f_{M_2}f_{M_3} = \frac{\mathbf{F_{k,l+1}F_{k-1,l}}F_{k+1,l-1}}{F_{k+1,l+1}\mathbf{F_{k,l}}F_{k-1,l-1}}$ where in \textbf{bold} we put the $F$'s corresponding to the peaks and valleys.

\begin{figure}[h!]
    \begin{subfigure}{0.48\textwidth}
        \centering
        \begin{tikzpicture}[scale=1,every node/.style={inner sep=2pt}]
  \node (M3) at (3,-1){$\bullet$};
  \node (M33) at ( 2,0){$\bullet$};
  \node  (M44) at ( 4,0){$\bullet$};
  \node  (M23) at ( 1,1){$\bullet$};
  \node (M34) at ( 3,1) {$\bullet$};
  \node  (M45) at ( 5,1){$\bullet$};
  \node  (M24) at (2,2){$\bullet$};
  \node  (M35) at (4,2){$\bullet$};
  \node  (M25) at (3,3){$\bullet$};
  {\footnotesize
  \node[blue] (U33) at ( 3,0) {$M_3$};
  \node[blue] (U23) at (2,1) {$M_4$};
  \node[blue] (U34) at (4,1) {$M_2$};
  \node[blue] (U24) at (3,2) {$M_1$};
}
{
\tiny
    \node[below] at (M3) {$F_{k+1,l-1}$};
    \node[below left] at (M33) {$F_{k,l-1}$};
    \node[below right] at (M44) {$F_{k+1,l}$};
    \node[left] at (M23) {$F_{k-1,l-1}$};
    \node[right] at (M34) {$F_{k,l}$};
    \node[right] at (M45) {$F_{k+1,l+1}$};
    \node[above left] at (M24) {$F_{k-1,l}$};
    \node[above right] at (M35) {$F_{k,l+1}$};
    \node[above] at (M25) {$F_{k-1,l+1}$};
}
  \draw[thick] (M33) -- (M3) -- (M44);
  \draw[thick] (M23) -- (M33) -- (M34) -- (M44) -- (M45) -- (M35) -- (M25) -- (M24) -- (M23);
  \draw[thick] (M24) -- (M34) -- (M35);
\end{tikzpicture}
        \caption{}
    \end{subfigure}
    \begin{subfigure}{0.48\textwidth}
        \centering
        \begin{tikzpicture}[scale=1,every node/.style={inner sep=2pt}]
  \node (M3) at (3,-1){$\bullet$};
  \node (M33) at ( 2,0){$\bullet$};
  \node  (M44) at ( 4,0){$\bullet$};
  \node  (M23) at ( 1,1){$\bullet$};
  \node[red] (M34) at ( 3,1) {$\bullet$};
  \node  (M45) at ( 5,1){$\bullet$};
  \node[red]  (M24) at (2,2){$\bullet$};
  \node[red]  (M35) at (4,2){$\bullet$};
  {\footnotesize
  \node[blue] (U33) at ( 3,0) {$M_2$};
  \node[blue] (U23) at (2,1) {$M_3$};
  \node[blue] (U34) at (4,1) {$M_1$};
}
{
\tiny
    \node[below] at (M3) {$F_{k+1,l-1}$};
    \node[below left] at (M33) {$F_{k,l-1}$};
    \node[below right] at (M44) {$F_{k+1,l}$};
    \node[left] at (M23) {$F_{k-1,l-1}$};
    \node[right] at (M34) {$\mathbf{F_{k,l}}$};
    \node[right] at (M45) {$F_{k+1,l+1}$};
    \node[above left] at (M24) {$\mathbf{F_{k-1,l}}$};
    \node[above right] at (M35) {$\mathbf{F_{k,l+1}}$};
}
  \draw[thick] (M33) -- (M3) -- (M44);
  \draw[thick] (M23) -- (M33) -- (M34) -- (M44) -- (M45) -- (M35);
  \draw[thick] (M23)--(M24) -- (M34) -- (M35);
  \draw[red, thick] (M24) -- (M34) -- (M35);
\end{tikzpicture}
        \caption{}
    \end{subfigure}
    \caption{Cancellations in a product of adjacent $f_{ij}$'s.}
    \label{fig: rectangle cancellation}
\end{figure}
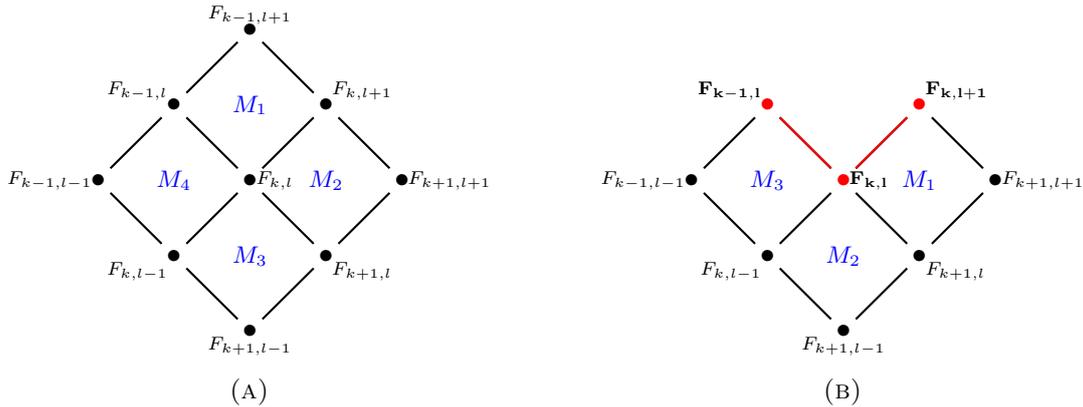

Adding the $f_X$ for the labels $(\Sigma k)$, or separately for $(k)$, in the shaved-off rectangle will cancel some of the $F$'s coming from peaks and valleys of the Dyck path. When we consider the left and right rectangles of $(i,j)$ we cancel in fact all of them. Additionally considering that for the right rectangle we need to add the $y$'s of the $(\Sigma k)$ yields exactly \eqref{eq: tellescopic canc}.

\smallskip

This shows that the $f_X(y)$ satisfy the $u$-equations. It is also clear that in the Laurent polynomial ring $\CC[U_\A^\circ] \subset \CC[u_X, u_X^{-1}]$ we can invert the map $\varphi^\circ: \CC[U_\A^\circ] \to R_\A^\circ$. This gives us a parametrization of the very affine variety $U_\A^\circ$. 

\medskip

Now, we want to see why the map $\varphi: \CC[U_\A] \to R_\A$ given by $u_X \to f_X$ is also an isomorphism. We proceed by induction on $n$. For $n=0,1,2$ this can be checked using computer algebra software. For the induction step we use the same technique as in \cite[Theorem 5.1]{clustconf}. Using Proposition~\ref{prop:fact} and the induction hypothesis we have that the induced map $\CC[U_\A]/(u_X) \to R_\A/(f_X)$ for any $X \in \I_\A$ is an isomorphism. Then, using the commutative algebra statement in \cite[Lemma A.1]{clustconf} it follows that $\varphi$ is an isomorphism. 
\end{proof}

As in \cite[Section 10]{stringy} the toric variety we want to consider is given by the normal fan of the Newton polytope of the $F$-polynomials used in the parametrization of Theorem~\ref{thm:param}. For every quotient algebra $\A$ we define the following polytope as a Minkowski sum
\begin{equation}\label{eq: polytope}
    P_\A := \sum_{M_{i,j} \text{ indecomposable of } \A} \Newt(F_{i,j}). 
\end{equation}
Denote its (inner) normal fan by $\Sigma_\A$. Recall from Section~\ref{sec:2} that if $D$ is the Dyck path corresponding to $\A$ we can read-off which $F$-polynomials appear in the sum \eqref{eq: polytope} by looking at the lattice points in the truncated grid corresponding to $D$. Let $\varepsilon_i$ be the standard basis vectors in $\RR^d$. The \emph{$g$-vector} of $X \in \I_\A$ is 
\begin{equation}\label{eq: gvector}
    g_X := \begin{dcases}
        \varepsilon_i & \text{if } X = (\Sigma i)\\
        -\varepsilon_i & \text{if } X = (i)\\
        \varepsilon_i - \varepsilon_{j} & \text{if } X = (i,j).
    \end{dcases}
\end{equation}

\begin{prop}\label{prop: rays of g vector fan}
    The rays of $\Sigma_\A$ are the $g$-vectors of $\A$.
\end{prop}

We defer the proof of Proposition~\ref{prop: rays of g vector fan} until Section~\ref{sec:poly} where we will study in more detail the polytopes $P_\A$. However, we require the result to introduce the main theorem of this section. 

\medskip

Given a Laurent polynomial we can obtain a piecewise linear function by replacing usual arithmetic by tropical arithmetic (where we take a trivial valuation on $\CC$). That is, we replace usual addition by tropical addition: $\min$, and usual multiplication by tropical multiplication: $+$. Explicitly from $p \in \CC[y,y^{-1}]$ we obtain a function $\trop(p):\RR^n \to \RR$ as follows
\[
    p = \sum_{a \in \ZZ^n} c_a y^a \quad  \mapsto \quad \trop(p) = \min(\langle a, Y\rangle: a \in {\rm supp}(p))
\]
where ${\rm supp}(p) := \{a\in \ZZ^n: c_a \neq 0 \}$ which is finite. Recall that the normal fan of $\Newt(p)$ is also obtained as the common domains of linearity of $\trop(p)$. As in \cite[Section 5]{clustconf} we obtain a description of the coordinate ring of $U_\A$ using the tropicalizations of our parametrization. The following result will be helpful.
 \begin{lem}\label{lem: param and gvectors}
        Let $X,Y \in \I_\A$. Then, $$\trop(f_X)(g_{Y}) = \delta_{X,Y}.$$
\end{lem}
\begin{proof}
Note that $\trop(F_{i,j})=\min\{0,Y_i,Y_i+Y_{i+1},\ldots,Y_i+\ldots +Y_j\}$. Then, for $l\leq k$
\[
    \trop(F_{i,j})(\varepsilon_l) = 0, \quad \trop(F_{i,j})(-\varepsilon_l) = -\mathbbm{1}_{i\leq l \leq j}, \quad \trop(F_{i,j})(\varepsilon_l-\varepsilon_k) = -\mathbbm{1}_{l<i\leq k \leq j}
\]
where $\mathbbm{1}$ is an indicator function.
 It follows that 
\begin{align}\label{eq: eval trop f}
    &\trop(f_{\Sigma i})(\varepsilon_l) = \trop(y_i)(\varepsilon_l)+\trop(F_{k_i,i+1})(\varepsilon_l)-\trop(F_{k_i,i})(\varepsilon_l)=\delta_{i,l} \nonumber\\ 
    &\trop(f_{\Sigma i})(-\varepsilon_l) = -\delta_{l,i} -\mathbbm{1}_{i+1\leq l \leq k_i}+\mathbbm{1}_{i\leq l \leq k_i} = 0 \nonumber\\
    &\trop(f_{\Sigma i})(\varepsilon_l-\varepsilon_j) = \delta_{i,l}-\delta_{i,j}-\mathbbm{1}_{l<i+1\leq j \leq k_i}+\mathbbm{1}_{l<i\leq j \leq k_i}.
\end{align}
We now check that if $(l,j) \in \I_\A$ then \eqref{eq: eval trop f} equals zero. Recall that $(l,j) \in \I_\A$ if and only if the square it labels is below the Dyck path corresponding to $\A$. In particular, the diamond it labels is strictly below the down-step labeled by $(\Sigma i)$. Concretely we get the inequalities: $l\leq i+1$ or $j \leq k_i$. A case by case computation yields indeed that \eqref{eq: eval trop f} is $0$. A similar calculation shows the result for the $f_X$ labeled by the other elements of $\I_\A$.
\end{proof}

\begin{prop}\label{prop: trop}
    The ring $R_\A$ is the subring of $\CC(y)$ generated by Laurent monomials in 
    \[
    \{y_1,\ldots, y_n\} \cup \{ F_{i,j}: M_{i,j} \text{ is an indecomposable module of } \A\}
    \]
    whose tropicalization is nonnegative on all of $\RR^n$.
\end{prop}
\begin{proof}
    Recall that the normal fan of $P_\A$ is given by the domains of linearity of the piecewise linear function $\trop(F_\A)$ where
    \[
        F_\A := \prod_{M_{i,j} \text{ indecomposable of } \A} F_{i,j}.
    \]
    Let $L(y)$ be a Laurent monomial as in the statement. As the factors $y_i$ do not affect the linearity of the tropicalization, it follows that the domains of linearity of $\trop(L)$ is a coarsening of the normal fan of $P_\A$. Hence, to determine the values of $\trop(L)$ we only need to know its values in the rays of the normal fan of $P_\A$, i.e. the $g$-vectors of $\A$.
    By Lemma~\ref{lem: param and gvectors} we can write $L(y)$ as 
    \[L(y) = \prod_{X \in \I_\A} f_X^{\trop(L)(g_X)}.\]
   Then, $L(y) \in R_\A$ if and only if $\trop(L)(g_M)\geq 0$ for all $g$-vectors of $\A$ which is equivalent to $\trop(L)$ being nonnegative. Every element in $R_\A$ is a polynomial on the $f_X$. As every monomial on $f_X$ arises as a Laurent monomial on $\{y_i,F_{i,j}\}$ whose tropicalization is nonnegative, the statement follows.
\end{proof}

We have now build up all the necessary results to show the main result of this section where we naturally see $U_\A$ inside a toric variety. We refer to \cite{toric} for standard results about toric varieties.

\begin{thm}\label{thm:toric}
    $U_\A$ is isomorphic to the affine open $\{F_\A \neq 0\}$ inside the projective toric variety associated to $\Sigma_\A$, $\mathcal{X}_{\Sigma_\A}$.
\end{thm}
\begin{proof}
    By Proposition~\ref{prop: trop} we can apply directly the technique from \cite[Section 10]{stringy}. By \cite{tilting} the normal fan of $P_\A$ is unimodular and hence in particular $P_\A$ is very ample. Then, we can consider an embedding of the (abstract) toric variety $\mathcal{X}_{\Sigma_\A}$ into $\PP^{l-1}$, where $l$ is the number of lattice points in $P_\A$. In this embedding $H = \{F_\A = 0\} \subset X_{P_\A}$ is a hyperplane as every monomial in $F_\A$ corresponds to a lattice point. It is then a direct check that the coordinate ring of $X_{P_\A}\setminus H$ is isomorphic to $R_\A$. 
\end{proof}

With the isomorphim from Theorem~\ref{thm:toric} we can give a description of the boundary poset of the nonnegative part of our configuration spaces: $(U_{\A})_{\geq0}$. In particular, we get a nice combinatorial description of the intersections of the divisors from Proposition~\ref{prop:fact}.

\begin{cor}\label{cor:toric}
    The face lattice of $P_\A$ is isomorphic to the poset of boundary strata of $(U_{\A})_{\geq0}$.
\end{cor}
\begin{proof}
    As all $F$-polynomials are substraction free then so is $F_\A$. Hence, it does not vanish in the positive orthant and $(U_{\A})_{\geq0}$ is isomorphic to the positive part of the toric variety $(X_{P_\A})_{\geq 0}$.
\end{proof}

The toric language also gives us an alternative way of understanding the monomial maps from Section~\ref{sec:3} as iterative blow-down maps. By Theorem \ref{thm:surj} it is enough to look at the covering relations in the Dyck path poset and describe their corresponding monomial maps. Let $\A', \A$ be two linear Nakayama algebras with corresponding Dyck paths $D$ and $D'$ such that $D \leq D'$ is a covering relation. In particular, this means that $|\I_\A| = |\I_{\A'}|+1$.

\begin{prop}\label{prop: toric}
   Denote $\mathcal{H}=\mathbb{V}(F_{\A}) \subset \mathcal{X}_{\Sigma_\A}$ and $\mathcal{H}'=\mathbb{V}(F_{\A'}) \subset \mathcal{X}_{\Sigma_{\A'}}$. There is a toric blow-up $\pi:\mathcal{X}_{\Sigma_\A} \to \mathcal{X}_{\Sigma_{\A'}}$ such that $\pi^{-1}(\mathcal{H}') \subset \mathcal{H}$. In particular, $\pi$ restricts to a well defined map $\pi_U: U_\A \to U_{\A'}$.
\end{prop}
\begin{proof}
    By Corollary~\ref{cor: star subdivision} the existence of the blow-up map $\pi$ follows from standard toric geometry. Moreover, by Proposition~\ref{prop: one extra ray} the blow-up locus is the codimension two subvariety $V(\sigma) \subset \mathcal{X}_{\Sigma_{\A'}}$ corresponding to the cone $\sigma \in \Sigma_{\A'}$ spanned by $\varepsilon_i, -\varepsilon_j$. In the coordinates of $X_{P_{\A'}}$, $V(\sigma)$ is the linear space where all coordinates corresponding to the lattice points of the face $F_{\sigma} \subset P_{\A'}$ vanish. In particular the blowup map $\pi$ can be seen as a projection from $\PP^{l-1}$ to $\PP^{l'-1}$, with $l = |P_{\A} \cap \ZZ^d|$ and $l' = |P_{\A'} \cap \ZZ^d|$, where we only keep the coordinates of the lattice points of $P_{\A'}$. Hence, as $F_{\A'}$ divides $F_{\A}$ it follows that $\pi^{-1}(\mathcal{H}') \subset \mathcal{H}$.
\end{proof}

From the proof of Proposition~\ref{prop: toric} one can check that the blow-down map corresponding to $\pi$ coincides with the monomial map in Definition~\ref{dfn: monomial maps}. For this, we can verify the maps coincide in the affine chart corresponding to the two dimensional cone $\sigma$ which we are subdividing into two cones to go from $\Sigma_{\A'}$ to $\Sigma_{\A}$.

\section{Polytopes}\label{sec:poly}
In this section we give a fully combinatorial description of the polytopes which arise in our study of linear Nakayama algebras. Recall that associated to a quotient $\A=\CC A_n/I$ we have the polytope $P_\A$ defined in \eqref{eq: polytope}. We will now study them using the techniques of \cite[Section 3]{sumsimplFS}. 

\medskip

To every subset $G \subset \{0,\ldots,n\}$ we associate a $|G|-1$ dimensional standard simplex
\[\Delta_G = \{(x_0,\ldots,x_{n}) \in \RR^{n+1}: x_i \geq 0, \quad x_i=0 \text{ if } i \notin G, \quad \text{and } x_0+\ldots + x_{n}=1\}. \]
For a collection of subsets $\mathcal{G}$ we consider $P_{\mathcal{G}} = \sum_{G \in \mathcal{G}}\Delta_G$. The focus of \cite{sumsimplFS, sumsimplP} is when $\mathcal{G}$ are particularly nice families of subsets called building sets. In our setup, we work with collections that \emph{are not} building sets but we can still use some key observations from their results. 

\smallskip

Recall that $P_\A$ is the Minkowski sum of simplices corresponding to the Newton polytopes of $F_{i,j}=1+y_i+y_iy_{i+1}+\ldots+y_i\cdots y_j$. We start by making a change of coordinates which allows us to see each summand as a standard simplex. Consider the polynomial ring $\ZZ[x_0,\ldots,x_{n}]$ and the monomial map given by $x_i=y_1y_2\cdots y_i$ for $i=1,\ldots,n$ and $x_0=1$.~Then,
$$ x_{i-1}+\ldots + x_{j+1} = y_{1}\cdots y_{i-1} (1+y_i+y_iy_{i+1}+\ldots+y_i\cdots y_j).$$
This corresponds to a linear transformation in $\ZZ^n$ given by the $n\times n$ upper triangular matrix $T$ with all entries on and above the diagonal being one. For each $i=1,\ldots,n$ we consider the vector $b_i=(1,\ldots,1,0,\ldots,0)$ whose first zero entry is at position $i$. Applying the affine transformation $T-b_i$ to the simplex $\Delta_{[i-1,j]}$ (concretely we apply it to the last $n$ coordinates ignoring $x_{0}$) gives exactly $\Newt(F_{i,j})$. As an example, we draw the two different realizations for the pentagon $P_{\CC A_2}$ in Figure~\ref{fig:two pentagons}. To keep track of the two different realizations we keep the notation $\varepsilon_i$ for the standard basis vectors in $\RR^n$ ($y$ coordinates) and use $\epsilon_i$ when referring to the standard basis vectors in $\RR^{n+1}$ ($x$ coordinates).

\begin{figure}[h!]
    \begin{subfigure}{0.48\textwidth}
    \centering
        \begin{tikzpicture}
            \node (A) at (0,0) {$\bullet$};
            \node (B) at (0,1.3) {$\bullet$};
            \node (C) at (1.3,2.6) {$\bullet$};
            \node (D) at (2.6,2.6) {$\bullet$};
            \node (E) at (2.6,0) {$\bullet$};
            \draw[thick] (A) -- (B) -- (C) -- (D) -- (E) -- (A);
        \end{tikzpicture}
        \caption{}
    \end{subfigure}
    \begin{subfigure}{0.48\textwidth}\centering
        \begin{tikzpicture}
            \node (A) at (0.7071,1.2247) {$\bullet$};
            \node (B) at (1.4142,0.0) {$\bullet$};
            \node (C) at (0.7071,-1.2247) {$\bullet$};
            \node (D) at (-0.7071,-1.2247) {$\bullet$};
            \node (E) at (-2.1213,1.2247) {$\bullet$};
            \draw[thick] (A) -- (B) -- (C) -- (D) -- (E) -- (A);
        \end{tikzpicture}
        \caption{}
    \end{subfigure}
    \caption{A pentagon in $y$ coordinates (A) and $x$ coordinates (B).}
    \label{fig:two pentagons}
\end{figure}
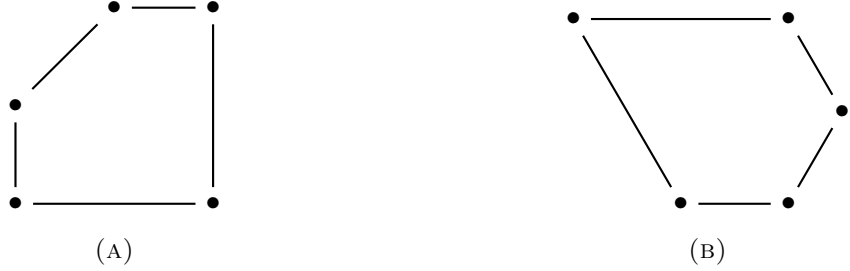

We consider again our triangular grid and add a row at the bottom. We now label the lattice points in the grid with intervals $[i,j]$ for $0\leq i \leq j \leq n$. The rows go bottom to top increasingly by interval size and horizontally left to right in lexicographic order, see Figure \ref{fig: poly grid ex}. Let $D$ be the Dyck path corresponding to the algebra $\A$. Drawing $D$ exactly as before in the grid with the new labelling will help us to describe the collection $\mathcal{G}$ that yields a polytope isomorphic to $P_\A$. Note that the new labelling we have chosen replaces $F_{i,j}$ by the interval $[i-1,j]$. Considering then the collection $\mathcal{G} = \{ G \subseteq \{0,\ldots,n\}: G \text{ is on or below } D \text{ and } |G|>1\}$, we obtain $P_\A\cong P_{\mathcal{G}}$.

\begin{figure}[h!]
    \centering
    \begin{tikzpicture}[every node/.style={inner sep=2pt}]
  \node (M11) at (-2,0) {$[0,1]$};
  \node (M22) at ( 0,0) {$[1,2]$};
  \node (M33) at ( 2,0) {$[2,3]$};
  \node (M44) at ( 4,0) {$[3,4]$};
  \node (M55) at ( 6,0) {$[4,5]$};
  \node (M12) at (-1,1) {$[0,2]$};
  \node (M23) at ( 1,1) {$[1,3]$};
  \node (M34) at ( 3,1) {$[2,4]$};
  \node (M45) at ( 5,1) {$[3,5]$};
  \node (M13) at (0,2) {$[0,3]$};
  \node (M24) at (2,2) {$[1,4]$};
  \node (M35) at (4,2) {$[2,5]$};
  \node (M14) at (1,3) {$[0,4]$};
  \node (M25) at (3,3) {$[1,5]$};
  \node (M15) at (2,4) {$[0,5]$};
 \node (P1) at (-3,-1) {$[0,0]$};
  \node (P2) at (-1,-1) {$[1,1]$};
  \node (P3) at (1,-1) {$[2,2]$};
  \node (P4) at (3,-1) {$[3,3]$};
    \node (P5) at (5,-1) {$[4,4]$};
      \node (P6) at (7,-1) {$[5,5]$};
    \draw[gray] (P1) -- (M11) -- (P2) -- (M22) -- (P3) -- (M33) -- (P4) -- (M44) -- (P5) -- (M55) -- (P6);
    \draw[gray] (M11) -- (M12) -- (M22) -- (M23) -- (M33) -- (M34) -- (M44) -- (M45) -- (M55);
    \draw[gray] (M12) -- (M13) -- (M23) -- (M24) -- (M34) -- (M35) -- (M45);
    \draw[gray] (M13) -- (M14) -- (M24) -- (M25) -- (M35);
    \draw[gray] (M14) -- (M15) -- (M25);
\end{tikzpicture}
    \caption{Grid for $\CC A_5$ labeled by intervals.}
    \label{fig: poly grid ex}
\end{figure}

\begin{rmk}
    Note that the collections $\mathcal{G}$ we obtain from Dyck paths in our grid are subsets of the collection of all intervals $\{ [i,j], 0\leq i \leq j \leq n\}$ characterized by the following two conditions
    \begin{itemize}
    \item $[i-1,i] \in \mathcal{G}$ for all $i=1,\ldots,n$;
    \item if $[i,j] \notin \mathcal{G}$, then no interval containing $[i,j]$ is in $\mathcal{G}$.
    \end{itemize}
    This implies, in particular, that the families $\mathcal{G}$ we consider yield connected simplicial complexes. Hence, by \cite[Remark 3.11]{sumsimplFS}, the polytopes $P_\A\cong P_{\mathcal{G}}$ have dimension $n$.
\end{rmk}

We now use these new coordinates to give a facet description of our polytopes. Recall that minimizing a linear form on a Minkowski sum is equivalent to minimizing it on each summand.  Hence, we obtain the following lemma, implicitly appearing already in \cite[Prop. 3.12]{sumsimplFS}.

\begin{lem}\label{lem: faces}
    Let $\mathcal{G}$ be a family of subsets of $\{0,\ldots,n\}$. For any $F \subseteq \{0,\ldots,n\}$, the linear form $\sum_{i \in F} x_i$ attains its minimum in the face of $P_{\mathcal{G}}$ given by 
    \begin{equation}\label{eq: faces}
        \sum_{G \in \mathcal{G}: G \subseteq F} \Delta_G + \sum_{G \in \mathcal{G}: G\setminus F \neq \emptyset} \Delta_{G\setminus F}. 
    \end{equation}
\end{lem}

We need one more combinatorial ingredient to fully describe the facets of our polytopes. Fix a Dyck path $D$ whose associated collection of intervals is $\mathcal{G}$. Let $D'$ be the Dyck path obtained by shifting $D$ vertically two rows down. That is, for every interval $[i,j]$ lying on $D$ the interval $[i+1,j-1]$ lies in $D'$, see Figure \ref{fig: lowered D ex} (A) for an example. 
\begin{figure}[h!]
\centering
    \begin{subfigure}{0.48\textwidth}
            \centering
    \begin{tikzpicture}[scale=0.65, transform shape]
  \node (M11) at (-2,0) {$[0,1]$};
  \node (M22) at ( 0,0) {$[1,2]$};
  \node (M33) at ( 2,0) {$[2,3]$};
  \node (M44) at ( 4,0) {$[3,4]$};
  \node (M55) at ( 6,0) {$[4,5]$};
  \node (M12) at (-1,1) {$[0,2]$};
  \node (M23) at ( 1,1) {$[1,3]$};
  \node (M34) at ( 3,1) {$[2,4]$};
  \node (M45) at ( 5,1) {$[3,5]$};
  \node (M13) at (0,2) {$[0,3]$};
  \node (M24) at (2,2) {$[1,4]$};
  \node (M35) at (4,2) {$[2,5]$};
  \node (M14) at (1,3) {$[0,4]$};
  \node (M25) at (3,3) {$[1,5]$};
  \node (M15) at (2,4) {$[0,5]$};
 \node (P1) at (-3,-1) {$[0,0]$};
  \node (P2) at (-1,-1) {$[1,1]$};
  \node (P3) at (1,-1) {$[2,2]$};
  \node (P4) at (3,-1) {$[3,3]$};
    \node (P5) at (5,-1) {$[4,4]$};
      \node (P6) at (7,-1) {$[5,5]$};
    \draw[gray] (P1) -- (M11) -- (P2) -- (M22) -- (P3) -- (M33) -- (P4) -- (M44) -- (P5) -- (M55) -- (P6);
    \draw[gray] (M12) -- (M22) -- (M23) -- (M33) -- (M34) -- (M44) -- (M45) -- (M55);
    \draw[gray] (M12) -- (M13) -- (M23) -- (M24) -- (M34) -- (M35) -- (M45);
    \draw[gray] (M13) -- (M14) -- (M24) -- (M25) -- (M35);
    \draw[gray] (M14) -- (M15) -- (M25);
  \draw[thick, blue] (-3,-3) -- (-2,-2) -- (P2) -- (0,-2) -- (P3) -- (M33) -- (M34)--(M44) -- (P5) -- (6,-2) -- (7,-3);
  \draw[thick, red] (P1) -- (M11) -- (M12) -- (M22) -- (M23) -- (M24) -- (M25)--(M35) -- (M45) -- (M55) -- (P6);
\end{tikzpicture}
    \caption{Dyck path \textcolor{red}{$D$} and its lowering \textcolor{blue}{$D'$}.}
    \end{subfigure}
    \begin{subfigure}{0.48\textwidth}
    \centering
    \begin{tikzpicture}[scale=0.65, transform shape]
  \node[blue] (M11) at (-2,0) {$[0,1]$};
  \node (M22) at ( 0,0) {$[1,2]$};
  \node[blue] (M33) at ( 2,0) {$[2,3]$};
  \node[blue] (M44) at ( 4,0) {$[3,4]$};
  \node[blue] (M55) at ( 6,0) {$[4,5]$};
  \node[blue] (M12) at (-1,1) {$[0,2]$};
  \node (M23) at ( 1,1) {$[1,3]$};
  \node[blue] (M34) at ( 3,1) {$[2,4]$};
  \node[blue] (M45) at ( 5,1) {$[3,5]$};
  \node[blue] (M13) at (0,2) {$[0,3]$};
  \node (M24) at (2,2) {$[1,4]$};
  \node[blue] (M35) at (4,2) {$[2,5]$};
  \node[blue] (M14) at (1,3) {$[0,4]$};
  \node[blue] (M25) at (3,3) {$[1,5]$};
  \node (M15) at (2,4) {$[0,5]$};
 \node[blue] (P1) at (-3,-1) {$[0,0]$};
  \node[blue] (P2) at (-1,-1) {$[1,1]$};
  \node[blue] (P3) at (1,-1) {$[2,2]$};
  \node[blue] (P4) at (3,-1) {$[3,3]$};
    \node[blue] (P5) at (5,-1) {$[4,4]$};
      \node[blue] (P6) at (7,-1) {$[5,5]$};
    \draw[gray] (P1) -- (M11) -- (P2) -- (M22) -- (P3) -- (M33) -- (P4) -- (M44) -- (P5) -- (M55) -- (P6);
    \draw[gray] (M12) -- (M22) -- (M23) -- (M33) -- (M34) -- (M44) -- (M45) -- (M55);
    \draw[gray] (M12) -- (M13) -- (M23) -- (M24) -- (M34) -- (M35) -- (M45);
    \draw[gray] (M13) -- (M14) -- (M24) -- (M25) -- (M35);
    \draw[gray] (M14) -- (M15) -- (M25);
  \draw[thick, red] (P1) -- (M11) -- (M12) -- (M22) -- (M23) -- (M24) -- (M25)--(M35) -- (M45) -- (M55) -- (P6);
\end{tikzpicture}
    \caption{Facet defining intervals in \textcolor{blue}{blue}.}
    \end{subfigure}
    \caption{}
    \label{fig: lowered D ex}
\end{figure}

\begin{prop}\label{prop: facet description}
    The facets of $P_{\mathcal{G}}$ are in one-to-one correspondence with the intervals 
    \[
        \mathcal{F}:=\{[0,i] : 0\leq i < n\} \cup \{[i,n] : 0< i \leq n\} \cup \{ \text{intervals on or below } D' \}.
    \]
    Moreover, the facet corresponding to the interval $[i,j]$ has inner normal vector $\epsilon_i + \ldots + \epsilon_j$.
\end{prop}
\begin{proof}
    We observe that the points of $P_{\mathcal{G}}$ are determined by 
    \begin{equation}\label{eq: facet ineqs}
        \sum_{i \in F} x_i \geq |\{G \in \mathcal{G}: G \subseteq F\}| \text{ for } G \subseteq \{0,\ldots,n\} \quad \text{and} \quad x_0+\ldots+x_{n}=|\mathcal{G}|.
    \end{equation}
    As $P_{\mathcal{G}}$ is a Minkowski summand of the $n$-dimensional permutahedron, then its facets are given by a subsets of the previous inequalities \eqref{eq: facet ineqs}. We will use Lemma~\ref{lem: faces} to determine which $G \subset \{0,\ldots,n\}$ yield factes, i.e. faces of codimension one. The left sum in \eqref{eq: faces} has dimension at most $|G|-1$. We get equality if $\{G \in \mathcal{G}: G \subseteq F\}$ represents a connected simplicial complex. This is the case if $G$ is an interval. The right sum in \eqref{eq: faces} has dimension at most $n-|F|$, with equality holding if and only if
    \begin{enumerate}
        \item $\{G\setminus F: G \in \mathcal{G} \text{ and } G \setminus F \neq \emptyset\}$ is a connected simplicial complex;
        \item and, there are $G_1,\ldots, G_r \in \mathcal{G}$ such that $\{0,\ldots,n\}\setminus F \subset G_1\cup\ldots \cup G_r$.
    \end{enumerate}
    Condition (2) is always satisfied for our collections since $\bigcup \mathcal{G} = \{0,\ldots,n\}$. If $F=[0,i]$ or $[i,n]$ then (1) is satisfied because their complement is an interval. If $F=[i,j]$ with $0<i<j<n$, (1) is satisfied if and only if $\{i-1,j+1\} \subseteq G$ for some $G \in \mathcal{G}$. As $[i-1,j+1]$ is at the bottom of the diamond below $[i,j]$ in our grid, this is the same as saying that $G$ is an interval below~$D'$.
\end{proof}

We can now prove Proposition~\ref{prop: rays of g vector fan} by applying a linear map to the normal vectors appearing in Proposition~\ref{prop: facet description}.

\begin{proof}[Proof of Proposition~\ref{prop: rays of g vector fan}]
    Note that the affine transformation we apply to the polytope $P_\A$ to obtain $P_{\mathcal{G}}$ turns into a linear transformation at the level of normal fans. Concretely we get a map between $\RR^{n+1}/(1,\ldots,1) \to \RR^n$ which in the standard basis (ignoring $\epsilon_{0}$ in $\RR^{n+1}$) is represented by the following matrix
    \[
       T^{-t} = \begin{bmatrix}
            1 & 0 & 0 & \ldots & 0 & 0 \\
            -1 & 1 & 0& \ldots & 0 & 0 \\
            0 & -1 & 1 & \ldots & 0 & 0 \\
            \vdots & \vdots & \vdots & \ddots & \vdots & \vdots \\
            0 & 0 & 0 & \ldots & 1 & 0\\
            0 & 0 & 0 & \ldots & -1 & 1
        \end{bmatrix}.
    \]
    Under this transformation the possible normal vectors of $P_{\mathcal{G}}$ map as follows.
    \begin{align*}
    &\text{For $i>0$:} &\epsilon_i + \ldots + \epsilon_{n} \mapsto  \varepsilon_{i}\\
        &\text{for $i>0$ and $j<n$:} &\epsilon_i + \ldots + \epsilon_j \mapsto  \varepsilon_i-\varepsilon_{j+1}\\
        &\text{for $j<n$:}&\epsilon_0 + \ldots + \epsilon_j \mapsto  -\varepsilon_{j+1}.
    \end{align*}
    By Proposition~\ref{prop: facet description} we have that the normal vectors of $P_{\mathcal{G}}$ after applying the affine transformation are $\sum_{i \in F} \epsilon_i$ where $F$ is $[0,i], [i,n]$ for some $i$ or it is in the lowering of $D$, the Dyck path of $\A$. In coordinates, this translates exactly to the $g$-vectors in \eqref{eq: gvector}.
\end{proof}

\begin{lem}\label{lem: compatibilityFacets}
    We say that $F,F' \in \mathcal{F}$ are \emph{$\mathcal{G}$-compatible} if for any $G \in \mathcal{G}$ such that $G \subset F\cup F'$ it holds that $G\subset F$ or $G\subset F'$. Two facets of $P_{\mathcal{G}}$, labelled by $F, F' \in \mathcal{F}$, intersect if and only if $F$ and $F'$ are $\mathcal{G}$-compatible.
\end{lem}
\begin{proof}
    Assume that $F$ and $F'$ are $\mathcal{G}$-compatible. Using \eqref{eq: faces} we have that the intersection of their facets is given by
    \[
        \sum_{G \in \mathcal{G}: G \subseteq F\cap F'} \Delta_G + \sum_{G \in \mathcal{G}: G\setminus (F\cup F') \neq \emptyset} \Delta_{G\setminus (F\cup F')} + \sum_{G \in \mathcal{G}: G\setminus F \neq \emptyset, G\setminus F' = \emptyset}  \Delta_{G\setminus F}+ \sum_{G \in \mathcal{G}: G\setminus F' \neq \emptyset, G\setminus F = \emptyset}  \Delta_{G\setminus F'}. 
    \]
    The $\mathcal{G}$-compatibility condition guarantees that this is a face of $P_{\mathcal{G}}$ as we are taking a Minkowski summand of $\Delta_G$ for each $G \in \mathcal{G}$. If $F$ and $F'$ are not $\mathcal{G}$-compatible then by definition there exists a $G \in \mathcal{G}$ such that the corresponding Minkowski summands appearing in the facets corresponding to $F$ and $F'$ are disjoint and hence cannot intersect.
\end{proof}

\begin{rmk}
An alternative way of describing the varieties from Section \ref{sec:3} is using the polytopes $P_{\mathcal{G}}$. We can label each $u$-variable by a facet of $P_{\mathcal{G}}$ and say two facets are \emph{compatible} if they intersect. Lemma~\ref{lem: compatibilityFacets} guarantees that the compatibility condition coincides with Definition~\ref{dfn: comp degree}, see Example~\ref{ex: compatibility} for a concrete example.

Given any polytope $P$ one can use the facets to describe a variety $U_P$ as the vanishing of a square system of equations as in \eqref{eq: ueqs}. However, the fact that $U_P$ is an irreducible variety of dimension $\dim(P)$ with a natural nonnegative part which is a``curvy'' version of $P$ is quite special. For example, the $u$-coordinates appearing in our configuration spaces for linear Nakayama algebras can de parameterized as very particular products of cross-ratios of points in $\PP^1$. The fact that $P_\A$ are somehow ``polytopes with cross-ratios'' yields the nice structure of $U_\A$. 
\end{rmk}

\begin{ex}\label{ex: compatibility}
   The grid in Figure~\ref{fig: lowered D ex} (A) corresponds to the algebra $\A = \CC A_5 / \langle \alpha_2 \alpha_1\rangle$ from Example~\ref{ex: A5 urels}. The facets corresponding to the labels $(2), (2,4) \in \I_\A$ in terms of intervals are respectively $[0,1]$ and $[2,3]$. The $\mathcal{G}$-incompatible intervals as in Lemma~\ref{lem: compatibilityFacets} are $[1,5],[2,2],[2,3],[2,4]$, $[2,5]$ for $[0,1]$ and $[0,1],[1,1],[1,2],[2,4],[3,4],[3,5],[4,5]$ for $[2,3]$. Translating back to the labels of $\I_\A$ yields exactly the $u$-relations in Figure~\ref{fig:urel2} and Figure~\ref{fig:urel24}.
\end{ex}

\begin{prop}\label{prop:simple}
    $P_{\mathcal{G}}$ is a simple polytope.
\end{prop}
\begin{proof}
    The same relabelling argument from \cite[Theorem 3.14]{sumsimplFS} works for our collections $\mathcal{G}$ even if they are not building sets. The key point is that (after relabelling) the facets adjacent to the vertex $v$ are $G_i=[i,n]$ which are facet defining for all $i=1,\ldots,n$.
\end{proof}

We now want to understand combinatorially what happens to the polytopes $P_{\mathcal{G}}$ as we let the family $\mathcal{G}$ increase one by one. Let $\mathcal{G}'$ be such that $\mathcal{G}=\mathcal{G}'\cup \{[i,j]\}$ and denote $D$ and $D'$ for the Dyck paths correspondign to $\mathcal{G}$ and $\mathcal{G}'$ respectively. Then $D$ is obtained by making exactly one valley of $D'$ a peak. We denote $\Sigma_{\mathcal{G}'}$ and $\Sigma_{\mathcal{G}}$ for the normal fans of $P_{\mathcal{G}'}$ and $P_{\mathcal{G}}$ respectively. 

\begin{prop}\label{prop: one extra ray}
    $\Sigma_{\mathcal{G}}$ has exactly one more ray than $\Sigma_{\mathcal{G}'}$, namely $r:=\epsilon_{i+1}+\ldots+\epsilon_{j-1}$. Moreover, $r$ is in a unique two dimensional cone of $\Sigma_{\mathcal{G}'}$.
\end{prop}
\begin{proof}
    By Proposition~\ref{prop: facet description}, $P_{\mathcal{G}}$ has one more facet than $P_{\mathcal{G}'}$. This extra facet is labeled by $[i+1,j-1]$ which corresponds to the normal ray $r=\epsilon_{i+1}+\ldots+\epsilon_{j-1}$. Note that $r$ can be written as sums of rays of $\Sigma_{\mathcal{G}'}$ in the following ways
    \begin{align*}
        &r = (\epsilon_{i+1}+\ldots+\epsilon_k)+(\epsilon_{k+1}+\ldots+\epsilon_{j-1}) &&\text{for some } k,\\
        &r = (\epsilon_{0}+\ldots+\epsilon_{j-1})+(\epsilon_{i+1}+\ldots+\epsilon_{n}).
    \end{align*}
    By Lemma~\ref{lem: compatibilityFacets} and Proposition~\ref{prop:simple} the rays labelled by $[i+1,k]$ and $[k+1,j-1]$ do not span a cone in $\Sigma_{\mathcal{G}'}$ whilst the rays corresponding to $[0,j-1]$ and $[i+1,n]$ do. 
\end{proof}

\begin{cor}\label{cor: star subdivision}
    $\Sigma_{\mathcal{G}}$ is the star subsivision of $\Sigma_{\mathcal{G}'}$ with respect to the ray $r$.
\end{cor}
\begin{proof}
    By Lemma~\ref{lem: compatibilityFacets} and Proposition~\ref{prop:simple} is clear that $\Sigma_{\mathcal{G}}$ refines $\Sigma_{\mathcal{G}'}$. We also know by Proposition~\ref{prop: one extra ray} that we are adding one extra ray in the refinement. As we are working with simplicial fans it is enough to check that both $\Cone(r,\epsilon_{0}+\ldots+\epsilon_{j-1})$ and $\Cone(r,\epsilon_{i+1}+\ldots+\epsilon_{n})$ are two dimensional cones of $\Sigma_{\mathcal{G}}$. As the sets that label the rays are $\mathcal{G}$-compatible it follows. 
\end{proof}

\bibliographystyle{plainurl}
\bibliography{references}

\end{document}